\documentclass[11pt,reqno]{amsart}

\usepackage{amsthm,amssymb,amstext,amscd,amsfonts,amsbsy,amsxtra,latexsym,amsmath,xcolor,mathrsfs,fancybox,mathtools}
\usepackage{fullpage}
\usepackage[english]{babel}
\usepackage[latin1]{inputenc}
\usepackage{cancel}
\usepackage[draft]{hyperref}
\usepackage{comment,enumitem}
\usepackage{mdframed}
\allowdisplaybreaks
\newcommand{\sm}[4]{\left(\begin{smallmatrix}#1&#2\\ #3&#4
\end{smallmatrix} \right)}
\newcommand{\field}[1]{\mathbb{#1}}
\newcommand{\N}{\field{N}}
\newcommand{\Z}{\field{Z}}
\newcommand{\R}{\field{R}}
\newcommand{\C}{\field{C}}
\newcommand{\Q}{\field{Q}}

\newcommand{\sgn}{\mathrm{sgn}}
\newcommand{\SL}{\operatorname{SL}}

\newcommand{\bea}{\begin{eqnarray}}
\newcommand{\eea}{\end{eqnarray}}
\newcommand{\be}{\begin {equation}}
\newcommand{\ee}{\end{equation}}

\newcommand{\spt}{\text{spt}}

\newcommand{\aqprod}[3]{\left(#1;#2\right)_{#3}}

\numberwithin{equation}{section}
\newtheorem{theorem}{Theorem}
\numberwithin{theorem}{section}
\newtheorem{lemma}[theorem]{Lemma}

\newtheorem{corollary}[theorem]{Corollary}

\theoremstyle{remark}

\newtheorem*{remark}{Remark}

\theoremstyle{definition}
\newtheorem*{definition}{Definition}

\renewenvironment{proof}[1][Proof]{\begin{trivlist} \item[\hskip \labelsep {\bfseries #1:}]}{\qed\end{trivlist}}

\title{On a modularity conjecture of Andrews, Dixit, Schultz, and Yee for a variation of Ramanujan's $\omega(q)$}
\author{Kathrin Bringmann, Chris Jennings-Shaffer, and  Karl Mahlburg}

\address{Mathematical Institute, University of Cologne, Weyertal 86-90, 50931 Cologne, Germany}
\email{kbringma@math.uni-koeln.de}

\address{Department of Mathematics, Oregon State University, Kidder Hall 368, Corvallis, Oregon, USA}
\email{jennichr@math.oregonstate.edu}

\address{Department of Mathematics, Louisiana State University, Baton Rouge, Louisiana, USA}
\email{mahlburg@mat.lsu.edu}

\begin{document}

\thanks{The research of the first author is supported by the Alfried Krupp Prize for Young University Teachers of the Krupp foundation and the research leading to these results receives funding from the European Research Council under the European Union's Seventh Framework Programme (FP/2007-2013) / ERC Grant agreement n. 335220 - AQSER}

\maketitle

\begin{abstract}
We analyze the mock modular behavior of $\bar{P}_\omega(q)$,
a partition function introduced by
Andrews, Dixit, Schultz, and Yee. This function arose in a study of
smallest parts functions related to classical third order mock theta functions, one of which
is $\omega(q)$. We find that the modular completion of $\bar{P}_\omega(q)$ is not simply
a harmonic Maass form, 
but is instead the derivative of a linear combination of products of various harmonic Maass forms and theta functions.
We precisely describe its behavior under modular transformations and find that
the image under the Maass lowering operator lies in a relatively simpler space.
\end{abstract}

\section{Introduction and statement of results}	
Andrews introduced the smallest parts partition function $\spt(n)$ in \cite{Aspt}, which enumerates the partitions of $n$ weighted by the multiplicity of their smallest parts. In the following years, a large volume of subsequent work has established its importance as a rich source for study, with many interesting examples of combinatorial and algebraic results \cite{Gar10, Ono}, connections to modular forms \cite{FO}, as well as generalized families of spt-functions \cite{DY, Gar11}. A particularly striking feature of the smallest parts function is found in the automorphic properties of its generating function \cite{Br}. Indeed, the generating function provides a natural example of a {\it mock theta function}, in the modern sense; i.e., the holomorphic part of a harmonic Maass form \cite{BruF04, Zag09}. The classical notion of a mock theta function was based on Ramanujan's last letter to Hardy \cite{Wat36}, and these functions are now understood in the framework of real-analytic modular forms thanks to Zwegers' seminal Ph.D. thesis \cite{ZwegersPhD}. In particular, in Theorem 4 of \cite{Aspt}, Andrews showed that
\begin{equation}
\label{E:sptAppell}
\sum_{n \geq 1} \spt(n) q^n = \frac{1}{(q;q)_\infty} \sum_{n \geq 1} \frac{nq^n}{1-q^n}
+ \frac{1}{(q;q)_\infty} \sum_{n \geq 1} \frac{(-1)^n q^{\frac{n(3n+1)}{2}}(1+q^n)}{(1-q^n)^2},
\end{equation}
where throughout the paper we use the standard $q$-factorial notation
$$
	(a)_n=(a;q)_n
	:=\prod_{j=0}^{n-1}\left(1-aq^j\right)$$ for $n\in\N_0\cup\{\infty\}.$
The key feature of \eqref{E:sptAppell} is that (up to rational powers of $q$) it expresses the generating function as a
derivative of a linear combination of theta functions and Appell-Lerch sums, which are the main components of Zwegers' work \cite{ZwegersPhD} (see Section \ref{S:Prelim} below).

In a recent paper \cite{ADY}, Andrews, Dixit, and Yee considered the question of constructing smallest parts partition functions directly from mock theta functions, and proved new results arising from the classical mock theta functions $\omega(q), \nu(q),$ and $\phi(q)$. Originally $\omega$ is defined as
$$
\omega(q):=\sum_{n\geq 0} \frac{q^{2n(n+1)}}{\left(q;q^2\right)_{n+1}^2};
$$
in Theorem 3.1 of \cite{ADY}, they proved the new representation
\begin{equation}
\label{E:qomega}
 P_\omega(q) := q \omega(q) = \sum_{n \geq 1} \frac{q^n}{\left(1 - q^n\right) \left(q^{n+1}; q\right)_n \left(q^{2n+2}; q^2\right)_\infty}.
\end{equation}
Basic combinatorial arguments show that the series \eqref{E:qomega} is the generating function for $p_\omega(n)$, which enumerates the partitions of $n$ in which the odd parts are less than twice the smallest part. Furthermore, this naturally gives a corresponding smallest parts function, which is a weighted count of each partition by the multiplicity of its smallest part:
\begin{equation*}
\sum_{n \geq 1} \spt_\omega(n) q^n := \sum_{n \geq 1} \frac{q^n}{\left(1 - q^n\right)^2 \left(q^{n+1}; q\right)_n \left(q^{2n+2}; q^2\right)_\infty}.
\end{equation*}
We note that this spt-function first appeared in the literature as $\spt_o^+(n)$ in \cite{Patkowski}, where Patkowski studied two Bailey pair identities that are directly related to Andrews' original smallest parts function. It also appears as $\spt_{C1}(n)$ in \cite{GJS}, where many examples of spt-functions were derived from Slater's extensive list of partition identities arising from Bailey pairs \cite{Sla}.

One of the main results arising from these various studies of $\spt_\omega(n)$ states that its generating function has a representation in terms of Appell-Lerch sums. Specifically, Theorem 1 in \cite{Patkowski} and Lemma 6.1 in \cite{ADY} provide two different proofs of the identity
\begin{equation}
\label{E:sptomega}
\sum_{n \geq 1} \spt_\omega(n) q^n =
\frac{1}{\left(q^2; q^2\right)_\infty} \sum_{n\geq 1} \frac{nq^n}{1-q^n}
+ \frac{1}{\left(q^2; q^2\right)_\infty} \sum_{n \geq 1} \frac{(-1)^n \left(1 + q^{2n}\right) q^{n(3n+1)}}{\left(1-q^{2n}\right)^2}.
\end{equation}
As discussed further below, this is essentially a mock modular form of weight $\frac{3}{2}$; note that this also yields a Hecke-type (indefinite) theta series representation by expanding the denominators using geometric series.

In \cite{ADSY}, Andrews, Dixit, Schultz, and Yee considered an overpartition analog of $p_\omega(n)$ and $\spt_\omega(n)$. In particular, they defined $\bar{p}_\omega(n)$ to be the number of overpartitions of $n$ such that all odd parts are less than twice the smallest part, and in which the smallest part is always overlined. The generating function in this case is
\begin{equation}\label{pbargen}
	\bar{P}_\omega(q):=\sum_{n\geq1} \bar{p}_\omega(n)q^n	 =\sum_{n\geq1}\frac{\left(-q^{n+1};q\right)_n\left(-q^{2n+2};q^2\right)_\infty}
{\left(1-q^n\right)\left(q^{n+1};q\right)_n\left(q^{2n+2};q^2\right)_\infty}q^n,
\end{equation}
with corresponding smallest parts function
\begin{equation}\label{E:ospt}
	\sum_{n\geq1} \bar{\spt}_\omega(n)q^n =\sum_{n\geq1}\frac{\left(-q^{n+1};q\right)_n\left(-q^{2n+2};q^2\right)_\infty}
{\left(1-q^n\right)^2\left(q^{n+1};q\right)_n\left(q^{2n+2};q^2\right)_\infty}q^n.
\end{equation}
The generating function in \eqref{E:ospt} previously appeared in \cite{JS}, where several additional families of spt-functions arising from Bailey pairs were studied systematically. Indeed, one of the functions defined in \cite{JS} is
\begin{equation*}
\sum_{n \geq 1} \spt_{G2}(n) q^n :=
\sum_{n \geq 1} \frac{q^n}{\left(1-q^n\right)^2 \left(q^{n+1}; q\right)^2_n \left(q^{2n+2}; q^2\right)_\infty \left(q^{4n+2}; q^4\right)_\infty},
\end{equation*}
and a short calculation shows that this is (termwise) equivalent to \eqref{E:ospt}.

Theorem 5.1 of \cite{ADSY} shows that, just like \eqref{E:sptAppell} and \eqref{E:sptomega}, the overpartition analog also has a representation in terms of Appell-Lerch sums (and hence well-understood modularity properties), namely
\begin{align*}
\sum_{n \geq 1} \bar{\spt}_\omega(n) q^n =
\frac{\left(-q^2; q^2\right)_\infty}{\left(q^2; q^2\right)_\infty}\sum_{n \geq 1} \frac{n q^n}{1-q^n}
+ 2 \frac{\left(-q^2; q^2\right)_\infty}{\left(q^2; q^2\right)_\infty} \sum_{n \geq 1} \frac{(-1)^n q^{2n(n+1)}}{\left(1-q^{2n}\right)^2}.
\end{align*}
However, there is a significant difference between $\bar{P}_\omega(q)$ and $P_\omega(q)$, as the modularity properties of $P_\omega(q)$
come directly from the mock theta function $\omega(q)$ in \eqref{E:qomega}. Indeed, Andrews, Dixit, Schultz, and Yee commented on the difficulty of relating $\bar{P}_\omega(q)$ to modular forms in \cite{ADSY}, and raised the question of whether this is even possible as Problem 1.

The main result of the present paper provides an affirmative answer to the question of Andrews, Dixit, Schultz, and Yee on the modularity properties of $\bar{P}_\omega$.
\begin{theorem}\label{mock}
The function $\bar{P}_\omega(q) + \frac14 - \frac{\eta(4\tau)}{2 \eta(2\tau)^2}$, where $\eta(\tau):=q^{\frac1{24}}\prod_{n\geq 1} (1-q^n)$ ($q:=e^{2\pi i\tau}$ throughout) is Dedekind's $\eta$-function, can be completed
to $\widehat{P}_\omega(\tau)$, which transforms like a weight $1$ modular form.
\end{theorem}
\begin{remark}
We recall that a mock modular form $f$ of weight $k$ has a modular completion $\widehat{f}$ with the property
that \begin{equation}
\label{xi}
\xi_k\left(\widehat{f}\right):= 2iv^k\ \overline{\frac{\partial}{\partial \overline{\tau}}\widehat f(\tau)}
\end{equation}
is a weakly holomorphic modular form of weight $2-k$ (the ``shadow'' of $f$).
In contrast, here $\xi_1(\widehat{P}_\omega)$ is not a weakly
holomorphic modular form. However,
$L(\widehat{P}_\omega(\tau)):=-2iv^2\frac{\partial}{\partial \overline{\tau}}\widehat{P}_\omega(\tau)$
lies in the space
$${v^{\frac32}\overline{S_{\frac32}(\Gamma_0(4),\chi_1)} H_{\frac12}(\Gamma,\chi_2)
\oplus v^{\frac12}M^!_{-\frac12}(\Gamma_0(4),\chi_3)\overline{M^!_{\frac12}(\Gamma_0(4),\chi_4)}}.
$$	
For this reason we call the function $\bar{P}_\omega(q) +\frac14-\frac12\frac{\eta(4\tau)}{\eta(2\tau)^2}$ a
{\it higher depth mock modular form}, as $L(\widehat{P}_\omega)$ is in a relatively simple space.
Although we use the term somewhat informally in this article, we say that {\it higher depth mock modular forms} are automorphic functions that are characterized and
inductively defined by the key property that their images under ``lowering operators'' essentially
lie in lower depth spaces. For example, the classical mock modular forms have ``depth 1'', since the
Maass lowering operator essentially yields a classical modular form.
Similarly, the main result of this paper shows that $\bar{P}_\omega(q) +\frac14-\frac12\frac{\eta(4\tau)}{\eta(2\tau)^2}$ is a mock modular form of depth 2.
For the precise details of the modular transformation properties of $\widehat{P}_\omega$,
see Theorem \ref{T:Phat}.
\end{remark}

The remainder of the paper is organized as follows.
In Section 2 we give a review of harmonic Maass forms, classical and mock Jacobi forms, and state two
$q$-series identities which we require below. In Section 3 we rewrite $\bar{P}_\omega$ as
an indefinite theta function in the form of a triple sum. In Section 4 we use the results of Section 3
to express $\bar{P}_\omega$ in terms of known mock modular forms, from which
Theorem \ref{mock} follows after a series of calculations.

\section*{Acknowledgments}
The authors thank the referees for their careful reading of the paper and many helpful comments.

\section{Preliminaries}
\label{S:Prelim}
\subsection{Harmonic Maass forms}
In this section we recall basic facts of harmonic Maass forms, first introduced by Bruinier and Funke in \cite{BruF04}. We begin with their definition.
\begin{definition}
For $k\in\frac12\Z$ , a {\it weight $k$ harmonic Maass form for a congruence subgroup  $\Gamma\subset\SL_2(\Z)$} (contained in $\Gamma_0(4)$ if $k\in\frac12+\Z$  )  is any smooth function $f:\mathbb H\rightarrow \C$ satisfying the following conditions:
\begin{enumerate}
	\item For all $\left(\begin{smallmatrix}
		a & b \\ c & d
	\end{smallmatrix} \right)\in \Gamma$, we have
	$$
	f\left(\frac{a\tau+b}{c\tau+d}\right)=
	\begin{cases}
	(c\tau+d)^k f(\tau)&\quad\text{ if } k\in\Z,\\
	\left(\frac{c}{d}\right)\varepsilon_d^{-2k}(c\tau+d)^k f(\tau)&\quad\text{ if } k\in\frac12+\Z.
	\end{cases}
	$$
	Here $(\frac{c}{d})$ is the generalized Legendre symbol in the sense of Shimura \cite{Sh} and
	$\varepsilon_d:=1$ if $d\equiv 1\pmod 4$ and $\varepsilon_d:=i$ if $d\equiv 3\pmod{4}$.
	
	\item We have $\Delta_k(f)=0$, where $\Delta_k$ is the {\it weight $k$ hyperbolic Laplacian} ($\tau = u+iv, u,v \in \R$ throughout)
	$$
	\Delta_k:=-v^2\left(\frac{\partial^2}{\partial u^2}+\frac{\partial^2}{\partial v^2}\right)+ikv\left(\frac{\partial}{\partial u}+i\frac{\partial}{\partial v}\right).
	$$
	\item There exists a polynomial $P_f(\tau)\in\C [q^{-1}]$ such that
	$$
	f(\tau)-P_f(\tau)=O(e^{-\varepsilon v})
	$$
	as $v\to \infty$ for some $\varepsilon >0$. Analogous conditions are required at all cusps.
\end{enumerate}
\end{definition}
\begin{remark}
Similarly one can define harmonic Maass forms of weight $k$ on $\Gamma$ with multiplier.
\end{remark}
Denote the space of such harmonic Maass forms of weight $k$ on $\Gamma$ by $H_k\left(\Gamma\right)$. Every $f\in H_k\left(\Gamma\right)$ has an expansion of the form
$$
f(\tau)=f^+(\tau)+f^-(\tau)
$$
with the {\it holomorphic part} having a $q$-expansion
\begin{equation*}
f^+(\tau)=\sum_{n\in\Q, n\gg -\infty}c^+_f(n)q^n
\end{equation*}
and the {\it non-holomorphic part} having an expansion of the form
\begin{equation*}
f^-(\tau)=\sum_{n\in\Q,n>0}c_f^-(n)\Gamma(1-k,4\pi nv)q^{-n}.
\end{equation*}
Here $\Gamma(s,v)$ is the {\it incomplete gamma function} defined, for $v>0$, as the integral
\begin{equation*}
	\Gamma(s,v) := \int_{v}^\infty t^{s-1}e^{-t} \, dt.
\end{equation*}
The function $f^+$ is called a {\it mock modular form}.
The non-holomorphic part of a harmonic Mass form may be written as a non-holomorphic Eichler integral of a classical modular form,
the so-called \emph{shadow} of the mock modular form. Indeed, if $f$ is a harmonic Maass form of weight $k$, then the shadow of
its holomorphic part can be recovered as $\xi_k (f)$, where $\xi_k$ is definied in \eqref{xi}. This operator maps harmonic Maass forms of weight $k$ to weakly holomorphic modular forms of dual weight $2-k$ and conjugated multipliers. It is related to the \textit{Maass lowering operator} $L:=-2iv^2\frac{\partial}{\partial\overline{\tau}}$, via $\xi_k=v^{k-2}\overline{L}$.
Harmonic Maass forms for which $f^{-}=0$ are weakly holomorphic modular forms and we denote the corresponding space by $M_k^{!}(\Gamma)$.
Additionally we denote the space of cusp forms by $S_k(\Gamma)$. For forms with multiplier, we use the notation
$H_k(\Gamma,\chi)$, $M_k^{!}(\Gamma,\chi)$, and $S_k(\Gamma,\chi)$.
We note that if $f(\tau)$ transforms like a modular form of weight $k$
with multiplier $\chi\begin{psmallmatrix}a&b\\c&d\end{psmallmatrix}$ for a subgroup $\Gamma$ of $\SL_2(\Z)$, then $f(-\bar{\tau})v^k$ transforms like a modular form of weight $-k$
with multiplier $\chi\begin{psmallmatrix}a&-b\\-c&d\end{psmallmatrix}$ on $\Gamma^\ast := \{\begin{psmallmatrix}a&b\\c&d\end{psmallmatrix}:\begin{psmallmatrix}a&-b\\-c&d\end{psmallmatrix}\in \Gamma \}$.
Note that $\Gamma^* = \Gamma$ for the congruence subgroups $\Gamma_0(N).$

\subsection{Jacobi forms and Zwegers' $\mu$-function}
In this section, we recall certain automorphic forms that we encounter in this paper. We start with classical Jacobi forms,
following Eichler and Zagier's seminal work \cite{EZ}.
\begin{definition}\label{harmjacdef} A {\it holomorphic Jacobi form of weight $k$ and index $m$} ($k, m \in \mathbb N$) on a subgroup $\Gamma \subseteq \textnormal{SL}_2(\mathbb Z)$ of finite index is a holomorphic function
$\varphi:\mathbb C \times \mathbb H \to
\mathbb C$ which, for all $\sm{a}{b}{c}{d} \in \Gamma$ and $\lambda,\mu \in \mathbb Z$, satisfies
 \begin{enumerate}\item
$\varphi\left(\frac{z}{c\tau + d};\frac{a\tau+b}{c\tau+d}\right) = (c\tau + d)^k e^{\frac{2\pi i mc z^2}{c\tau + d}} \varphi(z;\tau)$,
\item $\varphi(z + \lambda \tau + \mu;\tau) = e^{-2\pi i m (\lambda^2\tau + 2\lambda z)} \varphi(z;\tau)$,
\item $\varphi(z;\tau)$ has a Fourier expansion of the form $\sum_{n, r} c(n,r)q^n \zeta^r$ ($\zeta:=e^{2\pi iz}$ throughout) with $c(n,r)=0$ unless $n\geq r^2/4m$.
\end{enumerate}\end{definition}
\noindent Jacobi forms with multipliers and of half integral weight, meromorphic Jacobi forms, and weak Jacobi forms are defined similarly with obvious modifications made.

\begin{remark}
When specializing Jacobi forms to torsion points (i.e., $z\in\Q+\Q\tau$) one obtains classical modular forms (with respect to some congruence subgroup of SL$_2(\Z)$).
\end{remark}

A  special Jacobi form   used in this paper is \emph{Jacobi's theta function}, defined by
 \begin{equation*}
 \vartheta(z) =\vartheta(z;\tau):= \sum_{n\in\frac12+\Z}e^{\pi i n^2\tau+2\pi in\left(z+\frac12\right)},
\end{equation*}
where  here and throughout, we may omit the dependence of various functions on the variable $\tau$ if
the context is clear.
This function  is odd and well known to satisfy the following transformation law \cite{Knopp}.
\begin{lemma}\label{THETtrans}
 For $\lambda,\mu \in \mathbb Z$ and
$\sm{a}{b}{c}{d} \in \textnormal{SL}_2(\mathbb Z)$, we have that
\begin{eqnarray*}
 \vartheta(z+\lambda \tau+\mu;\tau)&=&(-1)^{\lambda+\mu}q^{-\frac{\lambda^2}{2}}\zeta^{-\lambda}
\vartheta(z;\tau),\label{tt1}\\ \vartheta\left(\frac{z}{c\tau+d};
\frac{a\tau+b}{c\tau+d}\right)&=&\psi\left(\begin{matrix}a&b\\c&d\end{matrix}\right)^3 (c\tau+d)^{\frac12}e^{\frac{\pi icz^2}{c \tau+d}}\vartheta(z;
\tau), \textnormal{\quad }
\end{eqnarray*}
where $\psi$ is the multiplier of $\eta$, i.e., $\eta(\frac{a\tau+b}{c\tau+d})=\psi\sm{a}{b}{c}{d}(c\tau + d)^{\frac{1}{2}}\eta(\tau)$. More explicitly, we have
\begin{equation*}
\psi\left(\begin{matrix} a&b\\c&d\end{matrix}\right)=
\begin{cases}
\left(\frac{d}{|c|}\right)e^{\frac{\pi i}{12}\left((a+d)c-bd\left(c^2-1\right)-3c \right)}
\qquad & \text{if } c \text{ is odd}, \\
\left(\frac{c}{d}\right)e^{\frac{\pi i}{12}\left(ac\left(1-d^2\right)+d(b-c+3)-3\right)} &\text{if } c \text{  is even}.
\end{cases}
\end{equation*}
\end{lemma}

Zwegers' mock Jacobi forms do not quite satisfy the elliptic and modular
transformations given in (1) and (2)
in the above definition of holomorphic Jacobi forms, but instead must be completed by adding a certain non-holomorphic function in order to satisfy suitable transformation laws.  To describe the simplest case, we define for $z_1, z_2 \in\C\setminus (\Z+\Z\tau)$ and $\tau\in\mathbb{H}$ the function
\[
\mu(z_1, z_2)=\mu(z_1, z_2; \tau):=\frac{e^{\pi iz_1}}{\vartheta(z_2; \tau)}\sum_{n\in\Z}
\frac{(-1)^ne^{2\pi inz_2}q^{\frac{n^2+n}{2}}}{1-e^{2\pi iz_1}q^n}
\]
and its completion
\begin{equation}\label{muhat}
\widehat{\mu}(z_1, z_2)=\widehat{\mu}(z_1, z_2; \tau):=\mu(z_1, z_2;\tau)+\frac{i}{2}R(z_1-z_2;\tau).
\end{equation}
Here the non-holomorphic part is given by $(\tau=u+iv,\ z=x+iy)$ the even function
\[
R(z)=R(z; \tau):=\sum_{n\in\frac12+\Z}\left(\sgn(n)-E\left(\left(n+\frac{y}{v}\right)\sqrt{2v}\right)\right)
(-1)^{n-\frac12}q^{-\frac{n^2}{2}}\zeta^{-n}
\]
with
\begin{equation}
\label{E:Edef}
E(w):=2\int_0^w e^{-\pi t^2}dt.
\end{equation}
For real arguments $w$, this can be expressed in terms of the incomplete gamma function, as
$$
E(w) = \sgn(w) \left(1 - \frac{1}{\sqrt{\pi}}\Gamma\left(\frac12, \pi w^2\right)\right).
$$
We also note that for certain values of $z$, $\tau\mapsto R(z;\tau)$ is in fact (weakly) holomorphic.
We have the following transformation laws.
\begin{lemma}[Zwegers \cite{ZwegersPhD}]\label{muhattranlem}
We have the symmetry relations
\begin{align*}
\widehat{\mu}(z_1,z_2)=\widehat{\mu}(z_2,z_1)
,\quad
\widehat{\mu}(-z_1,-z_2)=\widehat{\mu}(z_1,z_2).
\end{align*}
For $\left(\begin{smallmatrix} a & b \\ c & d \end{smallmatrix}\right)\in \mathrm{SL}_2(\Z)$ and $r_1,r_2,s_1,s_2 \in \mathbb Z$, we have
\begin{align*}
\widehat{\mu}\left(\frac{z_1}{c\tau+d}, \frac{z_2}{c\tau+d}; \frac{a\tau+b}{c\tau+d}\right)
&=\psi\left(\begin{matrix}a&b\\c&d\end{matrix}\right)^{-3}
(c\tau+d)^{\frac12}
e^{-\frac{\pi ic(z_1-z_2)^2}{c\tau+d}}\widehat{\mu}(z_1, z_2; \tau), \\
\widehat{\mu}\left(z_1 + r_1\tau + s_1,z_2 + r_2\tau + s_2\right) &= (-1)^{r_1+s_1+r_2+s_2} e^{\pi i (r_1-r_2)^2 \tau + 2\pi i (r_1-r_2)(z_1-z_2)} \widehat{\mu}(z_1,z_2).
\end{align*}

Moreover $R$ satisfies the elliptic shifts
\begin{align*}
R(z+1)&=-R(z),\\
R(z+\tau)&= -\zeta q^\frac12 R(z)+2\zeta^{\frac12} q^\frac38.
\end{align*}

For the function $\mu$, we have
\begin{align*}
\mu(z_1,z_2)&=\mu(z_2,z_1)
,\\
\mu(z_1+\tau,z_2)&=-e^{2\pi i(z_1-z_2)+\pi i\tau}\mu(z_1,z_2)-ie^{\pi i(z_1-z_2)+\frac{3\pi i\tau}{4}}.
\end{align*}
\end{lemma}

We note that when restricting $z_1$ and $z_2$ to torsion points, $\mu$ gives rise to a harmonic Maass form of weight $\frac{1}{2}$ (see Corollary 8.15 in \cite{BFOR}). The corresponding shadow is obtained from
\begin{equation}
\label{E:dtauR}
\frac{\partial}{\partial\overline{\tau}}R(a\tau+b)=\frac{i}{\sqrt{2v}}e^{-2\pi a^2v}\sum_{n\in\frac12+\Z}(-1)^{n+\frac12}(n+a)
e^{-\pi in^2\overline{\tau}-2\pi in(a\overline{\tau}+b)}.
\end{equation}
Thus the shadow of $\mu(z_1, z_2; \tau)$, for $z_j$ specialized to torsion points, may be written in terms of the weight 3/2 unary theta function ($a,b\in\Q$)
\[
\sum_{n\in a+\Z} n e^{2\pi inb} q^{\frac{n^2}{2}}.
\]
There is a similar identity for derivatives of $R$.
To be more precise, we compute
\begin{equation}
\label{E:dtaudzALT}
\frac{\partial}{\partial\overline{\tau}}\left[\frac{\partial}{\partial z} R(z)\right]_{z=a\tau+b}
=
	\frac{e^{-2\pi a^2v}}{2\sqrt{2v}}
	\sum_{n\in\frac{1}{2}+\mathbb Z}
		(-1)^{n-\frac{1}{2}} 	 \left(\frac{1}{v}+4\pi a(a+n) \right)e^{-\pi in^2\overline{\tau}-2\pi i n(a\overline{\tau}+b)}	.
\end{equation}


\subsection{Combinatorial results}

In our initial identities for $\bar{P}_\omega$, we need two standard $q$-series identities.
The first one is a finite version of the Jacobi triple product
identity \cite[p. 49 ex. 1]{A}.
\begin{lemma}\label{FJTP}
For $n\in\N_0$, we have that
\begin{align*}
	 \frac{\left(\zeta,\zeta^{-1}q\right)_n}{(q)_{2n}}
	&=
	\sum_{j=-n}^n
	\frac{(-1)^j \zeta^j q^{\frac{j(j-1)}{2}}}{(q)_{n-j}(q)_{n+j}}
	.
\end{align*}
\end{lemma}
 Secondly, we require Heine's transformation \cite[p. 359 eq. (III.2)]{GR}.
\begin{lemma}\label{Heine}
Suppose $|q|,|\zeta|<1$, and $|c|<|b|$. Then we have
\begin{align*}
	\sum_{n\ge0}
	\frac{(a,b)_{n} \zeta^n}{(c,q)_{n}}
	&=
	 \frac{\left(\frac{c}{b},b\zeta\right)_{\infty}}{(c,\zeta)_{\infty}}
	\sum_{n\ge0}
		\frac{\left(\frac{ab\zeta}{c},b\right)_{n} \left(\frac{c}{b}\right)^n}
	{(b\zeta,q)_{n}}
.
\end{align*}
\end{lemma}

\section{Representation in terms of indefinite theta functions}
In this section we prove a fundamental identity for $\bar{P}_\omega$ that provides a representation in terms of indefinite theta functions. In fact, we find a new representation for a $1$-parameter generalization of $\bar{P}_\omega$. Define
\[
\bar P_\omega(\zeta;q):=\frac{(q)_\infty}{\left(\zeta,\zeta^{-1}q\right)_\infty\left(-q;q^2\right)_\infty}\sum_{n\ge1} \frac{\left(\zeta,\zeta^{-1}q\right)_n \left(-q;q^2\right)_n }{(q)_{2n}} q^n;
\]
this also appeared in \cite{JS}.
It is not hard to see, comparing termwise to \eqref{pbargen}, that
$$
\bar P_\omega(q)=\bar P_\omega(1;q).
$$


We prove a double series representation for $\bar{P}_\omega(\zeta;q)$ that ultimately helps us identify the modularity properties of $\bar{P}_\omega(q)$.
\begin{theorem}\label{Pwz}
We have
\[
\bar{P}_\omega(\zeta;q)=\frac{1}{\left(\zeta,\zeta^{-1}q,q\right)_\infty} \sum_{j\ge1} \sum_{n\ge0} \frac{(-1)^{j+1}i^n\left(1-\zeta^j\right)\left(1-\left(\zeta^{-1}q\right)^j\right)q^{\frac{j(j+1)}{2}+n\left(j+\frac12\right)}}{1+iq^{j+n+\frac12}}.
\]
\end{theorem}
\begin{proof}
To prove this identity, we isolate the coefficient of $\zeta^j$ in
\begin{align*}
P^*_\omega(\zeta;q)
	&:=
	\left(\zeta,\zeta^{-1}q\right)_{\infty}\bar P_\omega(\zeta;q)
	,	
\end{align*}
transform this coefficient with standard $q$-series techniques, and
then sum over $j$. For convenience we use the notation that $[\zeta^j]F(\zeta)$ is the
coefficient of $\zeta^j$ in a series $F(\zeta)$. Noting that
$P^*_\omega(\zeta;q)$ is symmetric in $\zeta$ and $\zeta^{-1}q$, we have
$[\zeta^{-j}]P^*_\omega(\zeta;q)=q^{j}[\zeta^{j}]P^*_\omega(\zeta;q)$, and so we only need to determine the coefficients of the non-negative powers of $\zeta$.

By Lemma \ref{FJTP}, we have that
\begin{align*}
	P^*_\omega(\zeta;q)
	&=
	\frac{(q)_{\infty}}{\aqprod{-q}{q^2}{\infty}}
	\sum_{n\ge1}
	\aqprod{-q}{q^2}{n} q^n
	\sum_{j=-n}^n
	\frac{(-1)^j \zeta^j q^{\frac{j(j-1)}{2}} }
	{(q)_{n-j}(q)_{n+j}}
.
\end{align*}
From the above we see that the calculation of the coefficients of $\zeta$ slightly differs depending on whether
$j\ge 1$ or $j=0$.

For $j\ge 1$, we have that
\begin{align}
\notag	\left[\zeta^j\right]P^*_\omega(\zeta;q)
	&=
		\frac{(-1)^j q^{\frac{j(j-1)}{2}} (q)_{\infty}}{\aqprod{-q}{q^2}{\infty}}
		\sum_{n=j}^\infty
		\frac{ \aqprod{-q}{q^2}{n} q^n}
		{(q)_{n-j}(q)_{n+j}}
	=
		\frac{(-1)^j q^{\frac{j(j-1)}{2}} (q)_{\infty}}{\aqprod{-q}{q^2}{\infty}}
		\sum_{n\ge0}
		\frac{ \aqprod{-q}{q^2}{n+j} q^{n+j}}
		{(q)_{n}(q)_{n+2j}}\\
\notag 	&=
		\frac{(-1)^j q^{\frac{j(j+1)}{2}} (q)_{\infty} \aqprod{-q}{q^2}{j}}
		{\aqprod{-q}{q^2}{\infty}(q)_{2j}}
		\sum_{n\ge0}
		\frac{ \left(iq^{j+\frac{1}{2}}, -iq^{j+\frac{1}{2}}\right)_{n} q^{n}}
		{\left(q^{2j+1},q\right)_{n}}
	\\
\notag	&=
		\frac{(-1)^j q^{\frac{j(j+1)}{2}} (q)_{\infty} \aqprod{-q}{q^2}{j}
			 \left(iq^{j+\frac{1}{2}},-iq^{j+\frac{3}{2}}\right)_{\infty} }
		{\aqprod{-q}{q^2}{\infty} (q)_{2j} \left(q^{2j+1},q\right)_{\infty}}
		\sum_{n\ge0}
		\frac{ \left(1+iq^{j+\frac{1}{2}}\right) i^nq^{n\left(j+\frac{1}{2}\right)} }
		{1+iq^{j+\frac{1}{2}+n}}
	\\
\label{E:Heinej}
	&=
		\frac{(-1)^j q^{\frac{j(j+1)}{2}} }{ (q)_{\infty}}
		\sum_{n\ge0}
		\frac{ i^nq^{n\left(j+\frac{1}{2}\right)} }
		{1+iq^{j+\frac{1}{2}+n}}
,
\end{align}
where in the penultimate equality we apply Lemma \ref{Heine}
with $\zeta=q$, $a=iq^{j+\frac{1}{2}}$, $b=-iq^{j+\frac{1}{2}}$, and $c=q^{2j+1}$.

For $j=0$, we instead have that
\begin{align*}
	\left[\zeta^0\right]P^*_\omega(\zeta;q)
	&=
		 \frac{(q)_{\infty}}{\aqprod{-q}{q^2}{\infty}}
		\sum_{n\ge1}
		\frac{ \aqprod{-q}{q^2}{n} q^n}
		{(q)^2_{n}}
	=
		 \frac{(q)_{\infty}}{\aqprod{-q}{q^2}{\infty}}
		\sum_{n\ge0}
		\frac{ \aqprod{-q}{q^2}{n} q^n}
		{(q)^2_{n}}
		-
		 \frac{(q)_{\infty}}{\aqprod{-q}{q^2}{\infty}}
	\\
	&=
		\frac{1}{(q)_{\infty}}
	 \sum_{n\ge0}
		\frac{ i^nq^{\frac{n}{2}} }
		{1+iq^{\frac{1}{2}+n}}
		-
		 \frac{(q)_{\infty}}{\aqprod{-q}{q^2}{\infty}}
.	
\end{align*}
Here the final equality follows by observing that the sum is exactly the $j=0$ case from \eqref{E:Heinej}.

Summing over $j$ then gives that
\begin{align*}
	& P^*_\omega(\zeta;q)\\
	&=
		\frac{1}{(q)_{\infty}}
		\sum_{n\ge0}
		\frac{ i^nq^{\frac{n}{2}} }
		{1+iq^{\frac{1}{2}+n}}
		-
		 \frac{(q)_{\infty}}{\aqprod{-q}{q^2}{\infty}}
	+
		\frac{1}{(q)_{\infty}}
		\sum_{j\ge1}
		\left(\zeta^j+q^j\zeta^{-j}\right)	(-1)^j q^{\frac{j(j+1)}{2}}
		\sum_{n\ge0}
		\frac{ i^nq^{n\left(j+\frac{1}{2}\right)} }
		{1+iq^{j+\frac{1}{2}+n}}
	.
\end{align*}
However, we note that $P^*_\omega(1;q)=0$ and so
\begin{align*}
		\frac{1}{(q)_{\infty}}
		\sum_{n\ge0}
		\frac{ i^nq^{\frac{n}{2}} }
		{1+iq^{\frac{1}{2}+n}}
		-
		 \frac{(q)_{\infty}}{\aqprod{-q}{q^2}{\infty}}
		&=
		-
		\frac{1}{(q)_{\infty}}
		\sum_{j\ge1}
		\left(1+q^j\right)	(-1)^j q^{\frac{j(j+1)}{2}}
		\sum_{n\ge0}
		\frac{ i^nq^{n\left(j+\frac{1}{2}\right)} }
		{1+iq^{j+\frac{1}{2}+n}}
	.
\end{align*}
Thus
\begin{align*}
P^*_\omega(\zeta;q)
	&=
		\frac{1}{(q)_{\infty}}
		\sum_{j\ge1}
		\left(\zeta^j+q^j\zeta^{-j}-1-q^j\right) (-1)^j q^{\frac{j(j+1)}{2}}
		\sum_{n\ge0}
		\frac{ i^nq^{n\left(j+\frac{1}{2}\right)} }
		{1+iq^{j+\frac{1}{2}+n}}.
	\\
\end{align*}
Factoring gives the claim.
\end{proof}

Theorem \ref{Pwz} immediately leads to the following
indefinite theta series representation of $\bar{P}_\omega(q)$.

\begin{corollary}\label{Pwrep}
We have
\[
\bar P_\omega(q)=-\frac{1}{(q)^3_\infty}\left(\sum_{n,j,\ell\ge0}+\sum_{n,j,\ell<0} \right) j\left(1-q^j\right)(-1)^{j+n+\ell}q^{\frac{j(j+1)}{2}+2nj+2\ell j+4n\ell+n+\ell}.
\]
\end{corollary}
\begin{proof}
Taking $\zeta\to1$ in Theorem \ref{Pwz} and using that
\[
\lim_{\zeta\to1} \frac{\zeta^j-1}{1-\zeta}=-j
\]
gives that
\[
\bar P_\omega(q)=-\frac{1}{(q)^3_\infty}\sum_{j\ge1} \sum_{n\ge0} \frac{j\left(1-q^j\right)(-1)^j i^n q^{\frac{j(j+1)}{2}+n\left(j+\frac12\right)}}{1+iq^{j+n+\frac12}}.
\]
Note that throughout the following calculations we frequently also include the $j=0$ term if it gives a more convenient representation.

Expanding the geometric series then gives
\[
-(q)^3_\infty \bar P_\omega(q)=\sum_{j\ge1\atop{n,\ell\ge0}} j(-1)^j i^{n-\ell} \left(1-q^j\right)q^{\frac{j(j+1)}{2}+n\left(j+\frac12\right)+\ell\left(j+n+\frac12\right)}.
\]
Changing $(n,\ell)\mapsto(\ell,n)$, one sees that the contribution from $n\equiv \ell+1\pmod{2}$ vanishes. Splitting the sum according to the parity on $n$ and $\ell$ and then changing $(j,n,\ell)\mapsto (-j,-n,-\ell)$ in the second sum gives
\begin{align*}
-(q)^3_\infty \bar P_\omega(q)=&\sum_{j,n,\ell\ge0} j (-1)^{j+n+\ell} \left(1-q^j\right) q^{\frac{j(j+1)}{2}+2nj+2\ell j+4n\ell+n+\ell}\\
&+\sum_{j,n,\ell\ge1} j(-1)^{j+n+\ell} \left(1-q^j\right) q^{\frac{j(j-3)}{2}+2nj+2\ell j-n-\ell+4n\ell}\\
=&\left(\sum_{j,n,\ell\ge0} +\sum_{j,n,\ell<0}\right) j(-1)^{j+n+\ell} \left(1-q^j\right) q^{\frac{j(j+1)}{2}+2nj+2\ell j+n+\ell+4n\ell}.
\end{align*}
This yields the claim.
\end{proof}

\section{Proof of Theorem \ref{mock}}

Let
\[
\widehat{P}_\omega(\tau):=
\frac{i}{4\pi^2\eta(\tau)^6}\mathcal{F}'(0;\tau)^2
+\frac{e^{-\frac{\pi i}{4}}}{\pi^2}\frac{\eta(4\tau)}{4\eta(\tau)^3\eta(2\tau)^2}\mathcal{F}^{''}(0;\tau),
\]
where (recall $q=e^{2\pi i\tau}$ and $\zeta=e^{2\pi iz}$)
\begin{equation}\label{defineF}
\mathcal{F}(z; \tau):=q^{-\frac18}\zeta^{\frac12}\vartheta(z;\tau)\widehat{\mu}\left(z, \frac{\tau}{2}+\frac14;\tau\right).
\end{equation}
In this section, we prove that this function is a (non-holomorphic) modular form for
\[
\Gamma:=\left\{\begin{pmatrix} a&b\\c&d\end{pmatrix} \in\SL_2(\Z):4|c, \frac{c}{4}\equiv\frac{d-1}{2} \equiv b \pmod{2}\right\}.
\]
However, first we introduce a few auxiliary functions and determine their
modular properties. As we see below, these functions are related to
$\bar{P}_\omega(q)$ and $\widehat{P}_\omega(\tau)$.

Set
\begin{equation}\label{defineG}
G\left(z_1,z_2,z_3;\tau \right):=4 q^{-\frac18}\zeta_1^{-\frac12}\zeta_2^{\frac14}\zeta_3^{\frac14}
\left(\sum_{k>0\atop{\ell,n\ge0}}+\sum_{k\le0\atop{\ell,n<0}}\right)(-1)^k q^{\frac{k(k+1)}{2}+2k\ell+2kn+4\ell n}\zeta_1^k\zeta_2^\ell\zeta_3^n,
\end{equation}
where throughout this section we write $\zeta_j := e^{2 \pi i z_j}$ for the Jacobi parameters $z_j$.
It is not hard to see that (recall that we drop $\tau$-dependencies whenever they are clear from the context)
\begin{equation*}
G(z_1,z_2,z_3)=\sum_{0\le\alpha, \beta\le1}i^{-\alpha-\beta}F\left(z_1,\frac{z_2}{2}+\frac{\alpha}{2},\frac{z_3}{2}+\frac{\beta}{2}\right),
\end{equation*}
with
\[
F\left(z_1,z_2,z_3;\tau \right):=q^{-\frac18}\zeta_1^{-\frac12}\zeta_2^{\frac12}\zeta_3^{\frac12}
\left(\sum_{k>0\atop{\ell,n\ge0}}+\sum_{k\le0\atop{\ell,n<0}}\right)(-1)^k q^{\frac{k(k+1)}{2}+k\ell+kn+\ell n}\zeta_1^k\zeta_2^\ell\zeta_3^n.
\]

It is shown in Theorem 1.3 of \cite{BRZ} that
\begin{align}\label{EqBRZMockModular}
	F\left(z_1,z_2,z_3 \right)
	=
		 i\vartheta(z_1)\mu\left(z_1,z_2\right)\mu\left(z_1,z_3\right)
		-
		 \frac{\eta^3\vartheta(z_2+z_3)}{\vartheta(z_2)\vartheta(z_3)}
		\mu\left(z_1,z_2+z_3\right).
\end{align}
Below we plug in $z_1 = 0$, and it is important to note that $F$ has a
removable singularity at this value, even though the individual terms in
\eqref{EqBRZMockModular} may have poles. Indeed, this follows from the
evaluations
\begin{equation}\label{limit}
\vartheta^\prime(0) = -2\pi \eta^3
,\quad\quad
\lim_{z_1\rightarrow 0} \left(z_1\mu(z_1,z_2)\right) = \frac{-1}{2\pi i\vartheta(z_2)}.
\end{equation}

Recalling \eqref{muhat}, we define the modular completion of this function $F$ by
\[
\widehat{F}\left(z_1,z_2,z_3;\tau\right):=i\vartheta(z_1;\tau)\widehat{\mu}(z_1,z_2;\tau)\widehat{\mu}(z_1,z_3;\tau) -\frac{\eta(\tau)^3\vartheta(z_2+z_3;\tau)}{\vartheta(z_2;\tau)\vartheta(z_3;\tau)}\widehat{\mu}(z_1,z_2+z_3;\tau).
\]
With Lemmas \ref{THETtrans} and \ref{muhattranlem}, we find the following
Jacobi transformation laws hold.
\begin{lemma}
For $
\left(
\begin{smallmatrix}
a & b\\
c & d
\end{smallmatrix}
\right)
\in \SL_2(\Z)
$ and $n_j,m_j,\alpha,\beta\in\Z$, we have
\begin{align}
\label{E:hatFmod}
&\widehat{F}\left(\frac{z_1}{c\tau+d},\frac{z_2}{c\tau+d},\frac{z_3}{c\tau+d};\frac{a\tau+b}{c\tau+d}\right)
\\
	&\qquad\quad=
	\psi\left(\begin{matrix} a & b\\ c & d\end{matrix}\right)^{-3}(c\tau+d)^{\frac32}
	e^{\frac{\pi ic}{c\tau+d}\left(-z_1^2-z_2^2-z_3^2+2z_1z_2+2z_1z_3\right)}
	\widehat{F}\left(z_1,z_2,z_3;\tau\right),
	\notag\\
\notag
&\widehat{F}\left(z_1+n_1\tau+m_1,z_2+n_2\tau+m_2,z_3+n_3\tau+m_3\right)
\\
	&\qquad\quad=
	 (-1)^{n_1+m_1+n_2+m_2+n_3+m_3}\zeta_1^{n_1-n_2-n_3}\zeta_2^{n_2-n_1}\zeta_3^{n_3-n_1}q^{\frac{n_1^2}{2}+\frac{n_2^2}{2}+\frac{n_3^2}{2}-n_1n_2-n_1n_3}\widehat{F}(z_1,z_2,z_3),
	\notag\\
\label{Felliptic}
&\widehat{F}\left(z_1+n_1\tau+m_1,\frac{z_2}{2}+\frac{\alpha}{2}+n_2\tau+m_2,\frac{z_3}{2}+\frac{\beta}{2}+n_3\tau+m_3\right) \\
\notag
&\qquad\quad= (-1)^{n_1(1+\alpha+\beta)+n_2(1+\alpha)+n_3(1+\beta)+m_1+m_2+m_3} \zeta_1^{n_1-n_2-n_3} \zeta_2^{\frac{n_2}{2}-\frac{n_1}{2}} \zeta_3^{\frac{n_3}{2}-\frac{n_1}{2}} \\
\notag
&\qquad\qquad\times q^{\frac{n_1^2}{2} + \frac{n_2^2}{2}+\frac{n_3^2}{2} - n_1n_2-n_1n_3} \widehat{F}\left(z_1,\frac{z_2}{2}+\frac{\alpha}{2},\frac{z_3}{2}+\frac{\beta}{2}\right).
\end{align}
\end{lemma}
The modular completion of $G$ is now naturally defined as
\begin{equation}
\label{E:Ghat=Fhat}
\widehat{G}\left(z_1,z_2,z_3;\tau\right):= \sum_{0\le \alpha,\beta \le 1} i^{-\alpha-\beta} \widehat{F}\left(z_1,\frac{z_2}{2}+\frac{\alpha}{2},\frac{z_3}{2}+\frac{\beta}{2};\tau \right).
\end{equation}
We see below that modular completions introduce additional terms that can
be either holomorphic or non-holomorphic, which need to be identified separately.
By applying \eqref{E:hatFmod} and then \eqref{Felliptic}, we find modular
transformations for $\widehat{G}$. We also use \eqref{Felliptic} to obtain
elliptic transformations for $\widehat{G}$. The proofs are nothing more than
lengthy, but straightforward calculations reducing exponents and roots of
unity, and as such are omitted.
\begin{lemma}For
$\left(\begin{smallmatrix} a&b\\c&d\end{smallmatrix}\right) \in\Gamma$, we have
\begin{multline}\label{Ghatt}
\widehat{G}\left(\frac{z_1}{c\tau+d}, \frac{z_2}{c\tau+d}, \frac{z_3}{c\tau+d}; \frac{a\tau+b}{c\tau+d}\right)\\
=\psi\left(\begin{matrix}a&b\\c&d\end{matrix}\right)^{-3}(c\tau+d)^{\frac32}
e^{\frac{\pi ic}{c\tau+d}\left(-z_1^2-\frac{z_2^2}{4}-\frac{z_3^2}{4}+z_1z_2+z_1z_3\right)}\widehat{G}\left(z_1, z_2, z_3; \tau\right).
\end{multline}
For $m_1,n_1\in\Z$, $m_2,n_2,m_3,n_3\in2\Z$, and $n_1\equiv\frac{n_2}{2}\equiv\frac{n_3}{2}\pmod{2}$,
we have
\begin{multline}\label{E:Ghatell}
\widehat{G}(z_1+n_1\tau+m_1,z_2+n_2\tau+m_2,z_3+n_3\tau+m_3)\\
=(-1)^{n_1+\frac{n_2}{2}+\frac{n_3}{2}+m_1+\frac{m_2}{2}+\frac{m_3}{2}}\zeta_1^{n_1-n_2-n_3}\zeta_2^{\frac{n_2}{4}-\frac{n_1}{2}}\zeta_3^{\frac{n_3}{4}-\frac{n_1}{2}} q^{\frac{n_1^2}{2}+\frac{n_2^2}{8}+\frac{n_3^2}{8}-\frac{n_1n_2}{2}-\frac{n_1n_3}{2}}
\widehat{G}(z_1,z_2,z_3).
\end{multline}
\end{lemma}

Additionally, we require the following shifted $G$-functions
\begin{align*}
H(z;\tau):=q^{-\frac14} \zeta G\left(z,\tau+\frac12,\tau+\frac12;\tau\right),
\qquad\widehat{H}(z;\tau):=q^{-\frac14} \zeta\widehat{G}\left(z,\tau+\frac12,\tau+\frac12;\tau\right).
\end{align*}
Using \eqref{Ghatt} and \eqref{E:Ghatell}, we determine the modular transformation
of $\widehat{H}$ in the following lemma. Again, we omit the proof.
\begin{lemma}\label{GhatLemma}
For $\left(\begin{smallmatrix}a&b\\c&d\end{smallmatrix}\right)\in\Gamma$,
\begin{equation}
\label{E:Hmod}
\widehat{H}\left(\frac{z}{c\tau+d};\frac{a\tau+b}{c\tau+d} \right)=\psi\begin{pmatrix}a&b\\c&d\end{pmatrix}^{-3}\exp\left(-\frac{\pi i}{2}\left(ab+\frac{cd}{4}\right)\right) (c\tau+d)^{\frac32}  \exp\left(-\frac{\pi icz^2}{c\tau+d}\right)\widehat{H}\left(z;\tau\right).
\end{equation}
\end{lemma}

Before finally returning to $\widehat{P}_\omega$,
we additionally define the functions
\begin{align*}
f_1(\tau)&:= v^{\frac32}\eta(-4\overline{\tau})^3,
&f_2(\tau)&:= \frac{\mathcal{F}'(0;\tau)}{\eta(\tau)^3},
&f_3(\tau)&:= \frac{\eta(4\tau)}{\eta(2\tau)^2},
&f_4(\tau)&:= v^{\frac12}\frac{\eta(-2\overline{\tau})^5}{\eta(-\overline{\tau})^2\eta(-4\overline{\tau})^2}
\end{align*}
and the multipliers
\begin{align*}
\chi_1\begin{pmatrix}a&b\\c&d\end{pmatrix}
&:= \psi\begin{psmallmatrix}a&4b\\ \frac{c}{4}&d\end{psmallmatrix}
,&
\chi_2\begin{pmatrix}a&b\\c&d\end{pmatrix}
&:= \begin{cases}i^{\frac{c}{8}}(-1)^{\frac{d-1}{4}} & \mbox{ if } 8|c,\\ie^{-\frac{\pi ic}{16}} & \mbox{ if } 4||c, \end{cases}
\\
\chi_3\begin{pmatrix}a&b\\c&d\end{pmatrix}
&:=
\frac{\psi\begin{psmallmatrix}a&4b\\ \frac{c}{4}&d\end{psmallmatrix}}{\psi\begin{psmallmatrix}a&2b\\ \frac{c}{2}&d\end{psmallmatrix}^2}
,&
\chi_4\begin{pmatrix}a&b\\c&d\end{pmatrix}
&:=
\frac{\psi\begin{psmallmatrix}a&2b\\ \frac{c}{2}&d\end{psmallmatrix}^5}{\psi\begin{psmallmatrix}a&b\\ c&d\end{psmallmatrix}^2
\psi\begin{psmallmatrix}a&4b\\ \frac{c}{4}&d\end{psmallmatrix}^2}
.
\end{align*}
We note that $v^{-\frac32}f_1(-\overline{\tau})$ is a weight $\frac32$ cusp form,
$f_3(\tau)$ is a weight $-\frac12$ weakly holomorphic modular form,
and $v^{-\frac12}f_4(-\overline{\tau})$ is a weight $\frac12$ weakly holomorphic modular form,
on $\Gamma_0(4)$ with multipliers
$\chi_1$, $\chi_3$, and $\chi_4$, respectively.

\begin{theorem}\label{T:Phat}
The function $\widehat{P}_\omega$ has the following properties.\\
\noindent\textnormal{i)} We have, for $\gamma=\left(\begin{smallmatrix} a&b\\c&d\end{smallmatrix}\right) \in\Gamma $,
	\[
	 \widehat{P}_\omega\left(\frac{a\tau+b}{c\tau+d}\right)=e^{\frac{\pi ic}{8}}(c\tau+d)\widehat{P}_\omega(\tau).
	\]\\
\noindent\textnormal{ii)} The holomorphic part of $\widehat P_\omega$ (by which we mean the part that can be expressed as a $q$-series  $\sum_{n \geq 0} a(n) q^n$) is
	$$
	\bar{P}_\omega(q) +\frac14-\frac12\frac{\eta(4\tau)}{\eta(2\tau)^2}.
	$$
	
\noindent\textnormal{iii)} The Maass lowering operator acts as

	\[
	L\left(\widehat{P}_\omega\right)=\frac{e^{\frac{3\pi i}{4}}\sqrt{2}}{\pi} f_1f_2-\frac{e^{\frac{\pi i}{4}}}{2\sqrt{2}\pi}f_3f_4.
	\]
	Furthermore, the function $f_2$ is a harmonic Maass form of weight $\frac12$ with multiplier $\chi_2$ on $\Gamma$ and
	shadow $2\sqrt{2}e^{-\frac{\pi i}{4}}\pi\eta(4\tau)^3$.


\end{theorem}

\begin{proof}
We first relate $\bar{P}_\omega$ to a Jacobi-derivative of the indefinite theta-function
$G(z_1,z_2,z_3)$ defined in \eqref{defineG}.
Using the representation from Corollary \ref{Pwrep}, we obtain
\begin{align}\label{E:G+1/2}
\bar{P}_\omega(q)
=\frac{iq^{-\frac38}}{4(q)_\infty^3} \left(\left[\frac{\partial}{\partial \zeta} \left(\zeta^{\frac{1}{2}} G\left(z,\tau+\frac{1}{2},\tau+\frac{1}{2}\right)\right)\right]_{\zeta=1} \!\!\!- \left[\zeta\frac{\partial}{\partial \zeta} \left(\zeta^{\frac{1}{2}} G\left(z,\tau + \frac12,\tau + \frac12\right)\right)\right]_{\zeta=q}\right).
\end{align}

By the definition of $H(z;\tau)$, (\ref{E:G+1/2}) can be rewritten as
\begin{gather}\label{E:OverlinePToH}
\bar{P}_\omega(q)
=\frac{i}{4 \eta(\tau)^3}\left(
\left[\frac{\partial}{\partial\zeta}\left( \zeta^{-\frac12}H(z) \right)\right]_{\zeta=1}
-\left[\zeta\frac{\partial}{\partial\zeta}\left( \zeta^{-\frac12}H(z) \right)\right]_{\zeta=q}
\right)
\end{gather}
and the modular completion of (\ref{E:G+1/2}), which we show below is $\widehat{P}_\omega(\tau)$, becomes
\begin{align}
\label{E:Pw=H}
\frac{i}{4 \eta(\tau)^3}\left(\left[ \frac{\partial}{\partial \zeta}\left(\zeta^{-\frac{1}{2}}\widehat{H}(z)\right)\right]_{\zeta=1}-\left[\zeta \frac{\partial}{\partial \zeta}\left(\zeta^{-\frac{1}{2}}\widehat{H}(z)\right) \right]_{\zeta=q}\right).
\end{align}
Below we determine its holomorphic part.
Indeed, we see below that if $z$ is fixed, $\tau\mapsto\widehat{H}(z;\tau)-H(z;\tau)$ may be a mix of holomorphic and non-holomorphic functions.

We now regroup the terms from \eqref{E:Pw=H}, first noting that, for a function $f:\C\to\C$, we have
$$\left[\zeta\frac{\partial}{\partial\zeta}f(z)\right]_{\zeta=q}=\frac{1}{2\pi i}\left[\frac{\partial}{\partial z}f(z)\right]_{z=\tau}=\frac{1}{2\pi i}\left[\frac{\partial}{\partial z}f(z+\tau)\right]_{z=0}=\left[\frac{\partial}{\partial \zeta}f(z+\tau)\right]_{\zeta=1}.$$
Thus the second term in \eqref{E:Pw=H} can be expanded as
$$
-\left[\frac{\partial}{\partial\zeta}\left(q^{-\frac12}\zeta^{-\frac12}\widehat{H}(z+\tau)\right)\right]_{\zeta=1}
=-\frac12q^{-\frac12}\widehat{H}(\tau)
-\left[\frac{\partial}{\partial \zeta}\left(q^{-\frac12}\zeta^{-1}\widehat{H}(z+\tau)\right)\right]_{\zeta=1}.$$
We can now combine terms that satisfy the same modular transformations. For this, we set
\begin{align*}
\widehat{\mathcal{H}}_1(\tau):=-\frac12\left(\widehat{\mathcal{H}}_{11}(\tau)+\widehat{\mathcal{H}}_{12}(\tau)\right), \qquad
\widehat{\mathcal{H}}_2(\tau):=\widehat{\mathcal{H}}_{21}(\tau)-\widehat{\mathcal{H}}_{22}(\tau),
\end{align*}
where
\begin{alignat}{3}
\label{E:Hdefs}
\widehat{\mathcal{H}}_{11}(\tau)&:=\widehat{H}(0;\tau),& \qquad
& \widehat{\mathcal{H}}_{12}(\tau):=q^{-\frac12}\widehat{H}(\tau;\tau), \\
\widehat{\mathcal{H}}_{21}(\tau)&:=\left[\frac{\partial}{\partial \zeta} \widehat{H}(z;\tau)  \right]_{\zeta=1}, &
& \widehat{\mathcal{H}}_{22}(\tau):=\left[\frac{\partial}{\partial \zeta}\left(q^{-\frac12}\zeta^{-1}\widehat{H}(z+\tau;\tau)\right)\right]_{\zeta=1}. \notag
\end{alignat}
We see below that $\widehat{\mathcal{H}}_1$ and $\widehat{\mathcal{H}}_2$, respectively, satisfy modular transformations of weights $\frac32$ and $\frac52.$
In a similar fashion, we define the functions
\begin{align*}
\mathcal{H}_1(\tau) &:= -\frac{1}{2}\left(H(0;\tau)+q^{-\frac12}H(\tau;\tau)\right),
\\
\mathcal{H}_2(\tau) &:= \left[\frac{\partial}{\partial \zeta}H(z;\tau)  \right]_{\zeta=1}
		+\left[\frac{\partial}{\partial \zeta}\left(q^{-\frac12}\zeta^{-1}H(z+\tau;\tau)\right)\right]_{\zeta=1}.
\end{align*}
We note that in this notation \eqref{E:OverlinePToH} becomes
$\bar{P}_\omega=\frac{i}{4\eta^3}(\mathcal{H}_1+\mathcal{H}_2),$
and its modular completion is given by
$\frac{i}{4\eta^3}(\widehat{\mathcal{H}}_1+\widehat{\mathcal{H}}_2)$.
Again, we must show that this modular completion is in fact $\widehat{P}_\omega$.

We next prove modularity of $\widehat{\mathcal{H}}_1$ and $\widehat{\mathcal{H}}_2$. Firstly,
from Lemma \ref{GhatLemma}, we have for $\left(\begin{smallmatrix}a&b\\c&d\end{smallmatrix}\right)\in\Gamma$
\begin{align}
\label{E:H11mod}
&\widehat{\mathcal{H}}_{11}\left(\frac{a\tau+b}{c\tau+d} \right)=\psi\begin{pmatrix}a&b\\c&d\end{pmatrix}^{-3}e^{-\frac{\pi i}{2}\left(ab+\frac{cd}{4}\right)}(c\tau+d)^{\frac32}\widehat{\mathcal{H}}_{11}(\tau),\\
\label{E:H21mod}
&\widehat{\mathcal{H}}_{21}\left(\frac{a\tau+b}{c\tau+d} \right)
=\psi\begin{pmatrix}a&b\\c&d\end{pmatrix}^{-3}e^{-\frac{\pi i}{2}\left(ab+\frac{cd}{4}\right)} (c\tau+d)^{\frac52}\widehat{\mathcal{H}}_{21}(\tau).
\end{align}
The other two functions from \eqref{E:Hdefs} also satisfy modular transformations, which we see by rewriting them. A direct calculation,
with \eqref{E:Ghat=Fhat} and \eqref{Felliptic}, yields
\begin{equation}
\label{shiftM}
\widehat{H}(z+\tau)=q^{\frac{1}{4}}\zeta^2\sum_{0\leq\alpha,\beta\leq 1}{i^{\alpha+\beta}\widehat{F}\left(z,\frac{\tau}{2}+\frac{1}{4}+\frac{\alpha}{2},\frac{\tau}{2}+\frac{1}{4}+\frac{\beta}{2}\right)}.
\end{equation}
One can verify the transformation formulas for
$q^{-\frac14} \zeta \widehat{F}(z,\frac{\tau}{2}+\frac{1}{4}+\frac{\alpha}{2},\frac{\tau}{2}+\frac{1}{4}+\frac{\beta}{2})$
to find that the function
$q^{-\frac12}\zeta^{-1}\widehat{H}(z+\tau)$
satisfies the same modular transformation as in Lemma \ref{GhatLemma}.
Thus the analogue of \eqref{E:H11mod} also holds for
$\widehat{\mathcal{H}}_{12}$, and \eqref{E:H21mod} holds for $\widehat{\mathcal{H}}_{22}$.

Using \eqref{muhat}, the modular completion of $F$ requires the additional terms in $\widehat{F} - F$, which equal
\begin{align}
\notag
R^*(z_1,z_2,z_3;\tau) := & -\frac12\vartheta(z_1;\tau)\mu(z_1,z_2;\tau)R(z_1-z_3;\tau) -\frac12\vartheta(z_1;\tau)R(z_1-z_2;\tau)\mu(z_1,z_3;\tau)\\
& -\frac{i}{4}\vartheta(z_1;\tau)R(z_1-z_2;\tau)R(z_1-z_3;\tau)
-\frac{i\eta(\tau)^3\vartheta(z_2+z_3;\tau)}{2\vartheta(z_2;\tau)\vartheta(z_3;\tau)}R(z_1-z_2-z_3;\tau).
\label{E:R*def}
\end{align}
We use this to help
identify the excess terms in the modular completion $\widehat{\mathcal{H}}_1.$ In particular, plugging into the definition of $\widehat{H}$ as well as \eqref{shiftM}, we have
\begin{equation}\label{EqH1HatMinusH1}\begin{split}
&-2\left(\widehat{\mathcal{H}}_1-\mathcal{H}_1\right)\\
&=\sum_{0\le\alpha,\beta\le 1}i^{-\alpha-\beta}\left(
q^{-\frac14}R^*\left(0,\frac{\tau}{2}+\frac{1}{4}+\frac{\alpha}{2},\frac{\tau}{2}+\frac{1}{4}+\frac{\beta}{2}\right)
+q^{\frac14}R^*\left(\tau,\frac{\tau}{2}+\frac{1}{4}+\frac{\alpha}{2},\frac{\tau}{2}+\frac{1}{4}+\frac{\beta}{2}\right)
\right).
\end{split}\end{equation}
Recalling \eqref{limit}, we compute that
\begin{equation}
\label{E:R0z2z3}
R^*(0,z_2,z_3)= \frac{i}{2}\eta^3 \left(\frac{R(z_2)}{\vartheta(z_3)}+\frac{R(z_3)}{\vartheta(z_2)} - \frac{\vartheta(z_2+z_3)}{\vartheta(z_2)\vartheta(z_3)}R(z_2+z_3) \right).
\end{equation}
Similarly,
a calculation with Lemmas \ref{THETtrans} and \ref{muhattranlem} yields
\begin{multline}
\label{E:Rtauz2z3}
R^*(\tau,z_2,z_3)= -q^{\frac12}\zeta_2^{-1}\zeta_3^{-1}R^*\left(0,z_2,z_3\right)
+ iq^{\frac38} \zeta_2^{-\frac12} \zeta_3^{-\frac12} \eta^3
\left(\frac{\zeta_3^{-\frac12}}{\vartheta(z_3)} + \frac{\zeta_2^{-\frac12}}{\vartheta(z_2)} - \frac{\vartheta(z_2+z_3)}{\vartheta(z_2) \vartheta(z_3)}\right).
\end{multline}

We now determine whether $\widehat{\mathcal{H}}_1-\mathcal{H}_1$ adds any additional holomorphic terms to $\bar{P}_\omega$.
The first term of \eqref{E:Rtauz2z3} combined with (\ref{E:R0z2z3}) contributes
the following to \eqref{EqH1HatMinusH1}
\begin{align*}
&\sum_{0\le\alpha,\beta\le 1}i^{-\alpha-\beta}\left(
q^{-\frac14}R^*\left(0,\frac{\tau}{2}+\frac{1}{4}+\frac{\alpha}{2},\frac{\tau}{2}+\frac{1}{4}+\frac{\beta}{2}\right)
+(-1)^{\alpha+\beta}q^{-\frac14}R^*\left(0,\frac{\tau}{2}+\frac{1}{4}+\frac{\alpha}{2},\frac{\tau}{2}+\frac{1}{4}+\frac{\beta}{2}\right)
\right)
\\
&= 2q^{-\frac{1}{4}}\left(R^*\left(0,\frac{\tau}{2}+\frac{1}{4},\frac{\tau}{2}+\frac{1}{4} \right) - R^* \left(0,\frac{\tau}{2}+\frac{3}{4},\frac{\tau}{2}+\frac{3}{4} \right)\right) \\
&= iq^{-\frac{1}{4}}\eta^3\left(\frac{2R\left(\frac{\tau}{2}+\frac{1}{4}\right)}{\vartheta \left(\frac{\tau}{2}+\frac{1}{4} \right)}-\frac{2R\left(\frac{\tau}{2}+\frac{3}{4}\right)}{\vartheta \left( \frac{\tau}{2}+\frac{3}{4}\right)}-\frac{\vartheta \left(\tau + \frac{1}{2}\right)}{\vartheta \left(\frac{\tau}{2}+\frac{1}{4}\right)^2}R\left(\tau + \frac{1}{2}\right)+\frac{\vartheta \left(\tau + \frac{3}{2}\right)}{\vartheta \left(\frac{\tau}{2}+\frac{3}{4} \right)^2}R\left(\tau +\frac{3}{2}\right)\right).
\end{align*}
Noting that
\begin{align}
\notag
&\vartheta\left(\tau+\frac12\right)=-2q^{-\frac12}\frac{\eta(2\tau)^2}{\eta(\tau)},\quad \vartheta\left(\tau + \frac{3}{2}\right) = - \vartheta \left(\tau + \frac{1}{2}\right),  \\
\label{shiftR}
&\vartheta\left(\frac{\tau}{2} + \frac{3}{4}\right) = -i \vartheta\left(\frac{\tau}{2}+\frac{1}{4} \right), \quad \vartheta\left(\frac{\tau}{2}+\frac14\right)=e^{-\frac{3\pi i}{4}}q^{-\frac18}\frac{\eta (2\tau)^2}{\eta(4\tau)},\\
 \notag
&R\left(\tau+\frac12\right)=2iq^{\frac38}, \quad R\left(\frac{\tau}{2} +\frac14 \right)  = i R\left(\frac{\tau}{2} + \frac34\right) - 2 e^{-\frac{3 \pi i}{4}} q^\frac18, \quad
R\left(\tau + \frac{3}{2}\right) = -R\left( \tau + \frac{1}{2}\right),
\end{align}
we obtain that the previous expression equals
\begin{equation}
\label{E:Rterms}
-\frac{4i\eta(\tau)^3 \eta(4\tau)}{\eta(2\tau)^2} + \frac{8iq^{-\frac18} \eta(\tau)^2 \eta(4\tau)^2}{\eta(2\tau)^2}.
\end{equation}

The second and third terms from \eqref{E:Rtauz2z3} contribute (after switching the roles of $\alpha$ and $\beta$ in the third term)
\begin{align*}
	2e^{-\frac{\pi i}{4}}q^{-\frac{1}{8}}\eta^3
	\sum_{0\le \alpha,\beta\le 1}
	 \frac{(-1)^{\alpha}i^{\beta}}{\vartheta\left(\frac{\tau}{2}+\frac{1}{4}+\frac{\beta}{2}\right)}
=0.
\end{align*}


For the final term of \eqref{E:Rtauz2z3}, we note that $\vartheta(\tau+\frac12+\frac{\alpha+\beta}{2})=0$
if $\alpha\not\equiv\beta\pmod{2}$. Thus the corresponding sum simplifies to
\begin{align}
\notag
-q^{\frac{1}{8}}&\eta^3 \sum_{0 \leq \alpha, \beta \leq 1}\frac{(-1)^{\alpha+\beta} \vartheta \left(\tau + \frac{1}{2} + \frac{\alpha + \beta}{2}\right)}{\vartheta \left(\frac{\tau}{2}+\frac{1}{4}+\frac{\alpha}{2} \right) \vartheta \left(\frac{\tau}{2}+\frac{1}{4}+\frac{\beta}{2} \right)}
=-q^\frac18 \eta^3 \left(\frac{\vartheta\left(\tau + \frac12\right)}{\vartheta\left(\frac{\tau}{2} + \frac14\right)^2}
+ \frac{\vartheta\left(\tau + \frac32\right)}{\vartheta\left(\frac{\tau}{2} + \frac34\right)^2}\right)\\
&= -2q^\frac18 \eta^3 \frac{\vartheta\left(\tau + \frac12\right)}{\vartheta\left(\frac{\tau}{2} + \frac14\right)^2}
 = -\frac{4iq^{-\frac18} \eta(\tau)^2 \eta(4\tau)^2}{\eta(2\tau)^2}, \label{E:H12final}
\end{align}
where we use (\ref{shiftR}) to rewrite $\vartheta(\tau + \frac12)$ and
$\vartheta(\frac{\tau}{2} + \frac14)$ in terms of eta quotients.
Combining with \eqref{E:Rterms}, we then obtain
\begin{equation}
\label{E:H1hat-H1final}
\widehat{\mathcal{H}}_1(\tau)-\mathcal{H}_1(\tau)=\frac{2i\eta(\tau)^3 \eta(4\tau)}{\eta(2\tau)^2} - \frac{2iq^{-\frac18} \eta(\tau)^2 \eta(4\tau)^2}{\eta(2\tau)^2}.
\end{equation}

We next consider $\widehat{\mathcal{H}}_2-\mathcal{H}_2$, and note that
\begin{align}\label{EqH2HatMinusH2}
\widehat{\mathcal{H}}_2-\mathcal{H}_2
&=
\left[\frac{\partial}{\partial\zeta}\left( \sum_{0\le\alpha,\beta\le1} i^{-\alpha-\beta}\left(
q^{-\frac14}\zeta R^*\left( z, \frac{\tau}{2}+\frac14+\frac{\alpha}{2},\frac{\tau}{2}+\frac14+\frac{\beta}{2} \right)
\notag\right.\right.\right.\\&\left.\left.\left.\qquad\qquad\qquad
\quad-q^{\frac14} R^*\left( z+\tau, \frac{\tau}{2}+\frac14+\frac{\alpha}{2},\frac{\tau}{2}+\frac14+\frac{\beta}{2} \right)
\right)\right)\right]_{\zeta=1}
.
\end{align}
We again need to determine the holomorphic components of these terms for our specific choices of $z_j$.
We first evaluate the shifted term.
Using \eqref{E:R*def} and Lemmas \ref{THETtrans} and \ref{muhattranlem}, and then calculating the derivative directly, we find that
\begin{align}\notag
&\left[\frac{\partial}{\partial\zeta} R^*(z+\tau,z_2,z_3)\right]_{\zeta=1}
\\
\notag
&=-q^\frac12 \zeta_2^{-1} \zeta_3^{-1} \left[ \frac{\partial}{\partial \zeta} \left(\zeta R^*(z, z_2,z_3)\right)\right]_{\zeta = 1}
- \zeta_2^{-1} \zeta_3^{-\frac12} q^\frac38 \left[\frac{\partial}{\partial \zeta} \left(\zeta^\frac12 \vartheta(z) \mu(z,z_2)\right)\right]_{\zeta = 1} \\
\notag
& \quad - \zeta_3^{-1} \zeta_2^{-\frac12} q^\frac38 \left[\frac{\partial}{\partial \zeta} \left(\zeta^\frac12 \vartheta(z) \mu(z,z_3)\right)\right]_{\zeta = 1}
+ \frac{i}{2} \zeta_2^{-1} \zeta_3^{-\frac12} q^\frac38 \left[ \frac{\partial}{\partial \zeta} \left(\zeta^\frac12 \vartheta(z) R(z-z_2)\right)\right]_{\zeta=1} \\
\notag
&\quad+ \frac{i}{2} \zeta_3^{-1} \zeta_2^{-\frac12} q^\frac38 \left[ \frac{\partial}{\partial \zeta} \left(\zeta^\frac12 \vartheta(z) R(z-z_3)\right)\right]_{\zeta=1}
-2i q^\frac14 \zeta_2^{-\frac12} \zeta_3^{-\frac12} \left[ \frac{\partial}{\partial \zeta} \vartheta(z) \right]_{\zeta = 1} \\
\notag
&\quad -\frac{i}{2} \zeta_2^{-1} \zeta_3^{-\frac12} q^\frac38 \left[ \frac{\partial}{\partial \zeta} \left(\zeta^\frac12 \vartheta(z) R(z-z_2)\right)\right]_{\zeta = 1}
-\frac{i}{2} \zeta_3^{-1} \zeta_2^{-\frac12} q^\frac38 \left[ \frac{\partial}{\partial \zeta} \left(\zeta^\frac12 \vartheta(z) R(z-z_3)\right)\right]_{\zeta = 1} \\
\notag
& \quad+ i q^\frac14 \zeta_2^{-\frac12} \zeta_3^{-\frac12} \left[\frac{\partial}{\partial \zeta} \vartheta(z)\right]_{\zeta = 1}
- \frac{i q^\frac38 \zeta_2^{-\frac12} \zeta_3^{-\frac12} \eta^3 \vartheta(z_2 + z_3) }{2\vartheta(z_2) \vartheta(z_3)}.
\end{align}

After cancellation, we are left with the following terms:
\begin{equation}\label{E:H22nonhterms}\begin{split}
&-q^\frac12 \zeta_2^{-1} \zeta_3^{-1} \left[ \frac{\partial}{\partial \zeta} \left(\zeta R^*(z, z_2,z_3)\right)\right]_{\zeta = 1}
- \zeta_2^{-1} \zeta_3^{-\frac12} q^\frac38 \left[\frac{\partial}{\partial \zeta} \left(\zeta^\frac12 \vartheta(z) \mu(z,z_2)\right)\right]_{\zeta = 1} \\
& \quad - \zeta_3^{-1} \zeta_2^{-\frac12} q^\frac38 \left[\frac{\partial}{\partial \zeta} \left(\zeta^\frac12 \vartheta(z) \mu(z,z_3)\right)\right]_{\zeta = 1}
-i q^\frac14 \zeta_2^{-\frac12} \zeta_3^{-\frac12} \left[ \frac{\partial}{\partial \zeta} \vartheta(z) \right]_{\zeta = 1}
- \frac{i q^\frac38 \zeta_2^{-\frac12} \zeta_3^{-\frac12} \eta^3 \vartheta(z_2 + z_3) }{2\vartheta(z_2) \vartheta(z_3)}.
\end{split}\end{equation}

We now evaluate the contribution of each term of
\eqref{E:H22nonhterms} to \eqref{EqH2HatMinusH2}.
Recalling \eqref{limit}, we find that the fourth term of \eqref{E:H22nonhterms} sums  to
\begin{equation*}
i\eta^3\sum_{0\leq\alpha,\beta\leq 1}(-1)^{\alpha+\beta}=0.
\end{equation*}

For convenience, we define
\begin{align*}
f(z_j;\tau)
&:=
\left[ \frac{\partial}{\partial \zeta}\left(
\zeta^{\frac12}\vartheta(z;\tau)\mu(z,z_j;\tau)
\right)\right]_{\zeta=1}
=
\frac{\zeta_j^{\frac12}}{2}
\sum_{n\in\Z}\frac{(-1)^n(2n+1)q^{\frac{n(n+1)}{2}}}{1-\zeta_jq^n}
.
\end{align*}
By rewriting the second and third terms in terms of $f(z_j)$, we find that they also cancel completely, as
\begin{align*}
&-e^{\frac{\pi i}{4}}q^{-\frac18}
\sum_{0\le\alpha,\beta\le1}i^{-\alpha-\beta}\left(
(-1)^\alpha i^{-\beta} f\left(\frac{\tau}{2}+\frac{1}{4}+\frac{\alpha}{2}\right)
	+(-1)^\beta i^{-\alpha} f\left(\frac{\tau}{2}+\frac{1}{4}+\frac{\beta}{2}\right)
\right)
\\&=
-2e^{\frac{\pi i}{4}}q^{-\frac18}
\sum_{0\le\alpha,\beta\le1}i^\alpha (-1)^{\beta} f\left(\frac{\tau}{2}+\frac{1}{4}+\frac{\alpha}{2}\right)
=0.
\end{align*}

The fifth term of \eqref{E:H22nonhterms} does make a contribution to $\widehat{\mathcal{H}}_2 - {\mathcal{H}}_2$, which simplifies as
\begin{align*}
& -q^{\frac{1}{4}}\eta^3 \sum_{0\leq \alpha, \beta \leq 1}i^{-\alpha-\beta}\left(-\frac{i}{2}q^\frac{3}{8} e^{-\pi i \left(\frac{\tau}{2}+\frac{1}{4}+\frac{\alpha}{2} \right)} e^{-\pi i \left(\frac{\tau}{2} + \frac{1}{4} + \frac{\beta}{2}\right)} \frac{\vartheta\left(\tau + \frac{1}{2} +\frac{\alpha+\beta}{2}\right)}{\vartheta \left(\frac{\tau}{2} + \frac{1}{4} + \frac{\alpha}{2} \right)\vartheta \left(\frac{\tau}{2}+\frac{1}{4}+\frac{\beta}{2}\right)} \right) \\
& = \frac{2iq^{-\frac18} \eta(\tau)^2 \eta(4\tau)^2}{\eta(2\tau)^2},
\end{align*}
where we have the same cancellations as in \eqref{E:H12final}.

Finally, the first term of \eqref{E:H22nonhterms} combines with
$R^*(z,\frac{\tau}{2}+\frac{1}{4}+\frac{\alpha}{2},\frac{\tau}{2}+\frac{1}{4}+\frac{\beta}{2})$ from \eqref{EqH2HatMinusH2} to give
\begin{align}
\notag &q^{-\frac14}\left[\frac{\partial}{\partial\zeta}\left(\zeta\sum_{0\leq\alpha,\beta,\leq 1}\left(i^{-\alpha-\beta}-i^{\alpha+\beta}\right)R^*\left(z,\frac{\tau}{2}+\frac14+\frac{\alpha}{2},\frac{\tau}{2}+\frac14+\frac{\beta}{2}\right)\right)\right]_{\zeta=1} \\
\label{E:R*cancel}
=\ &-4i q^{-\frac14}\left[\frac{\partial}{\partial\zeta}\left(\zeta R^*\left(z,\frac{\tau}{2}+\frac14,\frac{\tau}{2}+\frac34\right)\right)\right]_{\zeta=1},
\end{align}
where in the final equality we use the symmetry $R^\ast(z, z_1, z_2)=R^*(z, z_2, z_1).$

To evaluate the above, recall that
\begin{align}\notag
&\left[\frac{\partial}{\partial\zeta}\left( \zeta R^*(z,z_2,z_3) \right)\right]_{\zeta=1}
\\
\notag
&=
\Bigg[\frac{\partial}{\partial\zeta}\left(
-\frac{1}{2}\zeta\vartheta(z)\mu(z,z_2)R(z-z_3)
-\frac{1}{2}\zeta\vartheta(z)\mu(z,z_3)R(z-z_2)
-\frac{i}{4}\zeta\vartheta(z)R(z-z_2)R(z-z_3)
\right.\\
\label{diffRstar}
&\qquad\qquad\left.
-\frac{i\eta^3\vartheta\left(z_2+z_3\right)}{2\vartheta(z_2)\vartheta(z_3)} \zeta R(z-z_2-z_3)
\right)\Bigg]_{\zeta=1}.
\end{align}
Plugging in our specific values of $z_2$ and $z_3$, we begin with the fourth term of
\eqref{diffRstar}. From the fact that $\vartheta$ vanishes at $\Z + \Z \tau$, this term vanishes.

Next we consider the third term of \eqref{diffRstar} when plugging into \eqref{E:R*cancel}. Noting that we need to differentiate $\vartheta$ (see \eqref{limit}) and using our specific values of $z_2, z_3$, along with the fact that $R$ is even, this becomes
\begin{equation}
\label{E:R*third}
-iq^{-\frac14} \eta^3 R\left(\frac{\tau}{2} + \frac14\right) R\left(\frac{\tau}{2} + \frac34\right).
\end{equation}

Finally, the first and second terms of \eqref{diffRstar} result in
\begin{align}
\nonumber
&2iq^{-\frac14}\Bigg(f\left(\frac{\tau}{2}+\frac14\right)R\left(\frac{\tau}{2}+\frac34\right)+f\left(\frac{\tau}{2}+\frac34\right)R\left(\frac{\tau}{2}+\frac14\right)\\
\label{writef}
&\quad +g\left(\frac{\tau}{2}+\frac14\right)\left[\frac{\partial}{\partial\zeta}\left(\zeta^{\frac12}R\left(z-\frac{\tau}{2}-\frac34\right)\right)\right]_{\zeta=1}+g\left(\frac{\tau}{2}+\frac34\right)\left[\frac{\partial}{\partial\zeta}\left(\zeta^{\frac12}R\left(z-\frac{\tau}{2}-\frac14\right)\right)\right]_{\zeta=1}\Bigg),
\end{align}
where
\[
g(z;\tau):=\zeta^{\frac12}\sum_{n\in\Z}\frac{(-1)^n q^{\frac{n(n+1)}{2}}}{1-\zeta q^n}.
\]

We now determine the holomorphic terms in \eqref{E:R*third} and \eqref{writef}.  We start with those that arise from not differentiating $R$.
Noting that
$$
\sgn\left(n - \frac12\right) - \sgn(n) = \begin{cases} 0 \qquad & \text{if } n \neq 0, \\
-1 & \text{if } n = 0,
\end{cases}
$$
we obtain
\begin{align*}
R\left(\frac{\tau}{2}+\frac14\right)
&=
\sum_{n\in\frac12+\Z}\left(\sgn\left(n\right)-E\left(\left(n+\frac12\right)\sqrt{2v}\right)\right)
	(-1)^{n-\frac12}e^{-\frac{\pi in}{2}}q^{-\frac{n(n+1)}{2}}
\\
&=
e^{\frac{\pi i}{4}}q^{\frac18}-e^{\frac{\pi i}{4}}q^{\frac18}
\sum_{n\in\Z}i^n\left(\sgn\left(n\right)-E\left(n\sqrt{2v}\right)\right)q^{-\frac{n^2}{2}}
=
e^{\frac{\pi i}{4}}q^{\frac18}+N_1(\tau)
,
\end{align*}
where in the last line we use that the sum over even integers $n$ vanishes, and where we set
\[
N_1(\tau):=  e^{-\frac{\pi i}{4}}q^{\frac18}
\sum_{n \in \Z} (-1)^n \left(\sgn(2n+1) - E\left((2n+1) \sqrt{2v}\right)\right) q^{-\frac12 (2n+1)^2}.
\]
Now $N_1$ is ``purely'' non-holomorphic, by which we mean that the coefficients of its $q$-expansion are incomplete Gamma functions with arguments in $v$ (cf. \eqref{E:Edef}).
Moreover, \eqref{shiftR} gives that
\begin{equation*}
R\left(\frac{\tau}{2}+\frac{3}{4}\right)
=-iN_1(\tau)+e^{\frac{3 \pi i}{4}}q^{\frac{1}{8}}.
\end{equation*}
We conclude that the holomorphic contribution of \eqref{E:R*third} to \eqref{EqH2HatMinusH2} is
$$
i\eta^{3}.
$$

Similarly, the holomorphic components from the first two terms in \eqref{writef} simplify as
\begin{align*}
2&e^{\frac{3\pi i}{4}}q^{-\frac18} \left(if \left(\frac{\tau}{2}+\frac{1}{4}\right)+f\left(\frac{\tau}{2}+\frac{3}{4}\right)\right)
=
-2iq^{\frac{1}{8}}\sum_{n \in \Z} \frac{\left(-1\right)^n\left(2n+1\right)q^{\frac{n(n+1)}{2}}}{1+q^{2n+1}}
\\
&=-iq^{\frac{1}{8}}\sum_{n \in \Z}(-1)^n(2n+1)q^{\frac{n(n+1)}{2}}=
-2i\eta^3.
\end{align*}

We rewrite the third term of \eqref{writef} as
\begin{align*}
& g\left(\frac{\tau}{2} + \frac14 \right) \left[\frac{\partial}{\partial \zeta} \left(- i \zeta^{-\frac12} R\left(-z-\frac{\tau}{2}-\frac14\right) +2e^{\frac{3\pi i}{4}}q^{\frac18}\right)\right]_{\zeta = 1} \\
& = -i g\left(\frac{\tau}{2} + \frac14 \right) \left[-\frac{\partial}{\partial \zeta} \left(\zeta^\frac12 R\left(z-\frac{\tau}{2}-\frac14\right)\right)\right]_{\zeta = 1}.
\end{align*}
The third and fourth terms (including the outside factor) now become
\begin{equation}\label{Rg} -2q^{-\frac14}\left[\frac{\partial}{\partial\zeta}\left(\zeta^{\frac12}R\left(z-\frac{\tau}{2}-\frac14\right)\right)\right]_{\zeta=1}\left(g\left(\frac{\tau}{2}+\frac14\right)-ig\left(\frac{\tau}{2}+\frac34\right)\right).
\end{equation}
To compute the derivative of $R$, we write
\begin{align*}
\zeta^{\frac12}R\left(z-\frac{\tau}{2}-\frac14\right)
=
e^{\frac{\pi i}{4}}q^{\frac18}+e^{\frac{\pi i}{4}}q^{\frac18}\sum_{n\in\Z} \left(\sgn(n)-E\left(\left(n+\frac{y}{v}\right)\sqrt{2v}\right)\right) i^{-n}q^{-\frac{n^2}{2}}\zeta^{-n}.
\end{align*}
Recalling \eqref{E:Edef}, we then find that
\begin{align*}
&\hspace{-15mm}\left[\frac{\partial}{\partial\zeta}\left( \zeta^{\frac12}R\left(z-\frac{\tau}{2}-\frac14\right)\right)\right]_{\zeta=1}
\\&=
	-e^{\frac{\pi i}{4}}q^{\frac18}\sum_{n\in\Z} i^{-n}n\left(\sgn(n)-E\left(n\sqrt{2v}\right)\right) q^{-\frac{n^2}{2}}
+
\frac{e^{\frac{\pi i}{4}}}{\sqrt{2v}\pi}q^{\frac18}\sum_{n\in\Z}i^{-n} e^{-\pi in^2\overline{\tau}}
\\
&=	-2e^{\frac{\pi i}{4}}q^{\frac18}\sum_{n\in\Z}(-1)^n n\left(\sgn(n)-E\left(2n\sqrt{2v}\right)\right) q^{-2n^2}
+
\frac{e^{\frac{\pi i}{4}}q^{\frac18}}{\pi\sqrt{2v}}
\sum_{n\in\Z} (-1)^n e^{-4\pi in^2\bar{\tau}}
\end{align*}
since for both summands the sums over odd integers $n$ vanish.
Thus \eqref{Rg} is purely non-holomorphic. We conclude that $\widehat{\mathcal{H}}_2-\mathcal{H}_2$ has the holomorphic part
\begin{equation}
\label{E:H2hat-H2final}
-i\eta(\tau)^3+2iq^{-\frac{1}{8}}\frac{\eta(\tau)^2\eta(4\tau)^2}{\eta(2\tau)^2}.
\end{equation}

Since \eqref{E:Pw=H} is exactly $\frac{i}{4 \eta^3} (\widehat{\mathcal{H}}_1 + \widehat{\mathcal{H}}_2)$, we can now find the holomorphic part by combining the above calculations above. In particular, we use the decomposition
$$
\widehat{\mathcal{H}}_1 + \widehat{\mathcal{H}}_2 = \left(\widehat{\mathcal{H}}_1 - \mathcal{H}_1\right) + \left(\widehat{\mathcal{H}}_2 - \mathcal{H}_2\right) + \mathcal{H}_1 + \mathcal{H}_2.
$$
Recalling \eqref{E:OverlinePToH}, and using \eqref{E:H1hat-H1final} and \eqref{E:H2hat-H2final}, the holomorphic part of $\widehat{\mathcal H}_1+\widehat{\mathcal H}_2$ equals
\begin{align}\label{E:holomorphicPart}
& \left(\frac{2i \eta(\tau)^3 \eta(4\tau)}{\eta(2\tau)^2} - \frac{2iq^{-\frac18} \eta(\tau)^2 \eta(4\tau)^2}{\eta(2\tau)^2}\right)
+ \left(-i\eta(\tau)^3 + \frac{2iq^{-\frac18} \eta(\tau)^2 \eta(4\tau)^2}{\eta(2\tau)^2}\right) - 4i\eta(\tau)^3 \bar{P}_\omega(q)
\notag \\
& = -4i\eta(\tau)^3 \left(\bar{P}_\omega(q) + \frac14 - \frac{\eta(4\tau)}{2\eta(2\tau)^2}\right).
\end{align}

Having found the relation between $\widehat{\mathcal{H}}_1+\widehat{\mathcal{H}}_2$ and $\bar P_\omega$, we next connect it to $\widehat{P}_\omega$. We start by showing
that $\widehat{\mathcal{H}}_1=0$. By (\ref{shiftM}), a direct calculation gives that
\begin{align*}
\widehat{\mathcal{H}}_1 & = q^{-\frac14} \left(\widehat{F}\left(0, \frac{\tau}{2} + \frac34, \frac{\tau}{2} + \frac34\right) - \widehat{F}\left(0, \frac{\tau}{2} + \frac14, \frac{\tau}{2} + \frac14\right)\right).
\end{align*}
We remember that $\widehat{F}(z, z_2, z_3)$ has a removable singularity at $z=0$, and
so by applying Lemmas \ref{THETtrans} and \ref{muhattranlem},
we obtain that
\begin{align*}
\widehat{F}\left(0, \frac{\tau}{2}+\frac34, \frac{\tau}{2}+\frac34\right)
 = \widehat{F}\left(0, \frac{\tau}{2}+\frac14;\frac{\tau}{2}+\frac14\right)
.
\end{align*}
From this we directly conclude that $\widehat{\mathcal{H}}_1=0$.

We next rewrite $\widehat{\mathcal{H}}_2$.
As a first simplification, we find
\begin{align*}
&\widehat{H}(z) - q^{-\frac12} \zeta^{-1} \widehat{H}(z + \tau)\\
=\ & q^{-\frac14} \zeta \sum_{0\leq\alpha,\beta\leq 1}\left(i^{-\alpha-\beta}-i^{\alpha+\beta}\right)
\widehat{F}\left(z,\frac{\tau}{2}+\frac14+\frac{\alpha}{2},\frac{\tau}{2}+\frac14+\frac{\beta}{2}\right) \\
=\ &-2i q^{-\frac14} \zeta\left(\widehat{F}\left(z,\frac{\tau}{2}+\frac34,\frac{\tau}{2}+\frac14\right)
+ \widehat{F}\left(z,\frac{\tau}{2}+\frac14,\frac{\tau}{2}+\frac34\right)\right)
=-4iq^{-\frac14}\zeta \widehat{F}\left(z,\frac{\tau}{2}+\frac14,\frac{\tau}{2}+\frac34\right),
\end{align*}
noting that $\widehat{F}(z,z_2,z_3)=\widehat{F}(z,z_3,z_2)$.
Thus we have that
\begin{equation*}
\widehat{\mathcal{H}}_2 = -4i q^{-\frac14} \left[\frac{\partial}{\partial \zeta}\left(\zeta\widehat{F}\left(z,\frac{\tau}{2}+\frac14,\frac{\tau}{2}+\frac34\right)\right)\right]_{\zeta = 1}.
\end{equation*}
%
We then compute
\begin{align*}
&\widehat F\left(z,\frac{\tau}{2}+\frac14,\frac{\tau}{2}+\frac34\right)
\\
=\ &i\vartheta(z)\widehat{\mu}\left(z,\frac{\tau}{2}+\frac14\right)\widehat{\mu}\left(z,\frac{\tau}{2}+\frac34\right)-\frac{\eta^3}{\vartheta\left(\frac{\tau}{2}+\frac14\right)\vartheta\left(\frac{\tau}{2}+\frac34\right)}\lim_{w\rightarrow 0}\left(\vartheta(w+\tau+1)\widehat{\mu}(z,w+\tau+1)\right).
\end{align*}
By \eqref{shiftR} and Lemmas \ref{THETtrans} and \ref{muhattranlem},
and further applying \eqref{limit}, we obtain
\begin{align}
\label{E:H2=F}
\widehat{\mathcal{H}}_2=\ &-4iq^{-\frac14}\left[\frac{\partial}{\partial \zeta}\left(
i\vartheta(z)\widehat{\mu}\left(z,\frac{\tau}{2}+\frac14\right)\widehat{\mu}\left(-z,\frac{\tau}{2}+\frac14\right)
-\frac{\eta^6}{\vartheta\left(\frac{\tau}{2}+\frac14\right)^2\vartheta(z)}\right)\right]_{\zeta=1} \notag\\
=\ &-4i \left[ \frac{\partial}{\partial \zeta} \left(\frac{1}{\vartheta(z)} \left\{-\mathcal{F}(z)\mathcal{F}(-z)-\frac{\eta^6q^{-\frac{1}{4}}}{\vartheta\left(\frac{\tau}{2}+\frac{1}{4}\right)^2} \right\} \right)\right]_{\zeta=1},
\end{align}
where $\mathcal F$ is defined in \eqref{defineF}.

In order to evaluate this derivative, we calculate the Laurent series of the interior expression and show that it has a removable singularity. First, \eqref{limit} implies that
\begin{equation*}
\frac{1}{\vartheta(z)}= -\frac{1}{2\pi \eta^3}\frac1z-\frac{\vartheta^{(3)}(0)}{24\pi^2 \eta^6}z+O\left(z^3\right).
\end{equation*}
The Taylor expansion of the bracketed terms in \eqref{E:H2=F} is
\begin{equation}
\label{E:FTaylor}
-\left(\mathcal{F}(0)^2 + \frac{\eta^6q^{-\frac{1}{4}}}{\vartheta\left(\frac{\tau}{2}+\frac{1}{4}\right)^2}\right)
-\left(\mathcal{F}(0)\mathcal{F}''(0)-\mathcal{F}'(0)^2 \right)z^2+O\left(z^3\right).
\end{equation}
We now show that the constant term of this expansion vanishes. Indeed, by definition
\begin{equation}
\label{E:calF}
\mathcal{F}(z) =q^{-\frac18} \zeta^\frac12 \vartheta(z) \mu\left(z, \frac{\tau}{2} + \frac14\right)
+ \frac{i}{2} q^{-\frac18}\zeta^\frac12 \vartheta(z) R\left(z - \frac{\tau}{2} - \frac14\right),
\end{equation}
and \eqref{limit} then implies that
\begin{equation}\label{E:F0}
\mathcal{F}(0)=-\frac{i\eta^3 q^{-\frac{1}{8}}}{\vartheta \left(\frac{\tau}{2}+\frac{1}{4}\right)}.
\end{equation}
Recalling \eqref{shiftR}, and plugging \eqref{E:FTaylor} back in to \eqref{E:H2=F}, we conclude that
\begin{equation*}
\widehat{\mathcal{H}}_2(\tau)=\frac{1}{\pi^2\eta(\tau)^3} \left(\mathcal{F}'(0)^2-e^{\frac{\pi i}{4}}\frac{\eta(\tau)^3\eta(4\tau)}{\eta(2\tau)^2}\mathcal{F}''(0)\right)
=
-4i\eta(\tau)^3\widehat{P}_\omega(\tau).
\end{equation*}

We are now in a position to prove the stated properties of $\widehat{P}_\omega$. In particular, we see
that i) follows from \eqref{E:H21mod} and \eqref{E:F0}, along with the fact that
\begin{gather*}
\psi\begin{pmatrix}a&b\\c&d\end{pmatrix}^{-6}e^{-\frac{\pi i}{2}\left(ab+\frac{cd}{4}\right)}
=e^{\frac{\pi ic}{8}}
\end{gather*}
on $\Gamma$, and ii) follows from (\ref{E:holomorphicPart}).

Thus we are left to show iii).
For this we use \eqref{E:dtauR} and \eqref{E:calF} to compute that (noting that $\vartheta$ must differentiated with respect to $\zeta$)
\begin{equation*}
\frac{\partial}{\partial\overline{\tau}}\mathcal{F}'(0)
=-\pi iq^{-\frac18} \eta(\tau)^3 \frac{\partial}{\partial\overline{\tau}}R\left(\frac{\tau}{2}+\frac14 \right)
=
\frac{\pi e^{\frac{\pi i}{4}} \eta(\tau)^3}{\sqrt{2v}}\sum_{n\in\Z}i^n ne^{-\pi i n^2\overline{\tau}}.
\end{equation*}
Noting that the sum over even integers $n$ vanishes, we obtain
\begin{equation}
\label{4.1}\
\frac{\pi i e^{\frac{\pi i}{4}}\eta(\tau)^3}{\sqrt{2v}} \sum_{n\in\Z}(-1)^n (2n+1) e^{-\pi i(2n+1)^2\overline{\tau}}
=\frac{\pi e^{\frac{3\pi i}{4}}\sqrt{2}}{\sqrt{v}} \eta(\tau)^3 \eta\left(-4\overline{\tau}\right)^3.
\end{equation}

Next we have
\[
\frac{\partial}{\partial\overline{\tau}}\mathcal{F}^{''}(0)=q^{-\frac18}\frac{i}{2}\left[\frac{\partial^2}{\partial z^2}
\frac{\partial}{\partial\overline{\tau}}\left(\zeta^{\frac12}\vartheta(z)R\left(z-\frac{\tau}{2}-\frac14\right)\right)\right]_{z=0}.
\]
Since $\vartheta$ is odd, we again need to differentiate $\vartheta$ in \eqref{E:calF} exactly once,  so that by using \eqref{limit},
\begin{equation}
\label{E:F''0}
\frac{\partial}{\partial\overline{\tau}}\mathcal{F}^{''}(0)
=
-2\pi i q^{-\frac18}\eta(\tau)^3\left[\frac{\partial}{\partial z}\frac{\partial}{\partial\overline{\tau}}\left(\zeta^{\frac12} R\left(z-\frac{\tau}{2}-\frac14\right)\right)\right]_{z=0}.
\end{equation}

The derivatives can be simplified and rewritten as
$$
\left[\frac{\partial}{\partial z}\frac{\partial}{\partial\overline{\tau}}\left(\zeta^{\frac12}R\left(z-\frac{\tau}{2}-\frac12\right)\right)\right]_{z=0}
=\pi i\frac{\partial}{\partial\overline{\tau}}R\left(-\frac{\tau}{2}-\frac12\right)
+\frac{\partial}{\partial\overline{\tau}}\left[
\frac{\partial}{\partial z}R(z)\right]_{z=-\frac{\tau}{2}-\frac12}
$$
We now apply both \eqref{E:dtauR} and \eqref{E:dtaudzALT}, finding that \eqref{E:F''0} equals
\begin{align*}
\frac{\pi\eta(\tau)^3}{\sqrt{2}v^{\frac32}}\sum_{n\in\Z}e^{-\pi in^2\overline{\tau}}
=
\frac{\pi\eta(\tau)^3}{\sqrt{2}v^{\frac32}}
\frac{\eta\left(-2\overline{\tau}\right)^5}{\eta\left(-\overline{\tau}\right)^2\eta\left(-4\overline{\tau}\right)^2}.
\end{align*}
We have now shown that the action of $L$ is as claimed.

We are left to prove that $\frac{\mathcal{F}'(0)}{\eta^3}$ is a harmonic Maass form of weight $\frac12$ whose shadow is
as claimed. The transformation follows by a lengthy calculation with Lemmas \ref{THETtrans} and \ref{muhattranlem}.
The claim about the shadow follows from \eqref{4.1} by multiplying by $2iv^{\frac{1}{2}}\eta(\tau)^{-3}$ and conjugating.

\end{proof}

\end{document}